\documentclass[11pt,reqno]{amsart}

\usepackage{verbatim}
\usepackage{graphicx}
\usepackage{amsmath,amssymb,amsfonts,euscript,enumerate,latexsym}
\usepackage{amsthm}
\usepackage{color,wrapfig}
\usepackage{pxfonts}
\usepackage{fancyhdr}
\usepackage{mathrsfs}
\usepackage{mathtools}
\usepackage{epsfig,subfigure}
\usepackage{marvosym}
\usepackage{txfonts}
\usepackage{hyperref}

\usepackage{etoolbox}
\makeatletter
\patchcmd{\@maketitle}
  {\ifx\@empty\@dedicatory}
  {\ifx\@empty\@date \else {\vskip3ex \centering\footnotesize\@date\par\vskip1ex}\fi
   \ifx\@empty\@dedicatory}
  {}{}
\patchcmd{\@adminfootnotes}
  {\ifx\@empty\@date\else \@footnotetext{\@setdate}\fi}
  {}{}{}
\makeatother

\newtheorem{thm}{Theorem}[section]
\newtheorem{cor}[thm]{Corollary}
\newtheorem{lemma}[thm]{Lemma}

\newtheorem{remark}[thm]{Remark}

\newcommand{\R}{{\mathbb{R}}}
\newcommand{\Z}{{\mathbb{Z}}}

\newcommand{\supp}{\operatorname{supp}}

\newcommand{\La}{\triangle}
\newcommand{\bs}{\backslash}

\newcommand{\1}{\partial}
\newcommand{\2}{\overline}
\newcommand{\3}{\varepsilon}

\begin{document}

\title[Singular limits and Properties of Degenerate equations]{Singular limits and properties of solutions of some degenerate elliptic and parabolic equations}

\author[K. M. Hui]{\textbf{Kin Ming Hui}}
\address{\textbf{Kin Ming Hui}:
Institute of Mathematics, Academia Sinica\\
Taipei, 10617, Taiwan, R.O.C.}
\email{kmhui@gate.sinica.edu.tw}

\author[S. Kim]{\textbf{Sunghoon Kim}}
\address{\textbf{Sunghoon Kim}:
Department of Mathematics, School of Natural Sciences, The Catholic University of Korea\\
43 Jibong-ro, Wonmi-gu, Bucheon-si, Gyeonggi-do, 14662, Republic of Korea}
\email{math.s.kim@catholic.ac.kr}



\begin{abstract}
Let $n\geq 3$, $0\le m<\frac{n-2}{n}$, $\rho_1>0$, $\beta>\beta_0^{(m)}=\frac{m\rho_1}{n-2-nm}$, $\alpha_m=\frac{2\beta+\rho_1}{1-m}$ and $\alpha=2\beta+\rho_1$. For any $\lambda>0$, we prove the uniqueness of radially symmetric solution $v^{(m)}$ of 
$\La(v^m/m)+\alpha_m v+\beta x\cdot\nabla v=0$, $v>0$, in $\R^n\setminus\{0\}$  which satisfies
$\lim_{|x|\to 0}|x|^{\frac{\alpha_m}{\beta}}v^{(m)}(x)=\lambda^{-\frac{\rho_1}{(1-m)\beta}}$ and obtain higher order estimates of $v^{(m)}$ near the blow-up point $x=0$. We prove that as $m\to 0^+$, $v^{(m)}$ converges uniformly in $C^2(K)$ for any compact subset $K$ of $\R^n\setminus\{0\}$ to the solution $v$ of $\La\log v+\alpha v+\beta x\cdot\nabla v=0$, $v>0$,
in $\R^n\bs\{0\}$, which satisfies $\lim_{|x|\to 0}|x|^{\frac{\alpha}{\beta}}v(x)=\lambda^{-\frac{\rho_1}{\beta}}$. We also prove that if the solution $u^{(m)}$ of $u_t=\Delta (u^m/m)$, $u>0$, in $(\R^n\setminus\{0\})\times (0,T)$ which blows up near $\{0\}\times (0,T)$ at the rate $|x|^{-\frac{\alpha_m}{\beta}}$ satisfies some mild growth condition on $(\R^n\setminus\{0\})\times (0,T)$, then as $m\to 0^+$, $u^{(m)}$  converges 
uniformly in $C^{2+\theta,1+\frac{\theta}{2}}(K)$ for some constant $\theta\in (0,1)$ and any compact subset $K$ of $(\R^n\setminus\{0\})\times (0,T)$ to the solution of $u_t=\La\log u$, $u>0$, in $(\R^n\setminus\{0\})\times (0,T)$. As a consequence of the proof we obtain existence of a unique radially symmetric solution $v^{(0)}$ of $\La \log v+\alpha v+\beta x\cdot\nabla v=0$, $v>0$, in $\R^n\setminus\{0\}$, which satisfies $\lim_{|x|\to 0}|x|^{\frac{\alpha}{\beta}}v(x)=\lambda^{-\frac{\rho_1}{\beta}}$. 
\end{abstract}

\maketitle
\vskip 0.2truein

\setcounter{equation}{0}
\setcounter{section}{0}

\section{Introduction}\label{section-intro}
\setcounter{equation}{0}
\setcounter{thm}{0}

Recently there is a lot of study on the equation \cite{DGL}, \cite{DS1},  \cite{FVWY}, [FW1--4], \cite{Hs3}, \cite{KL}, \cite{PS}, \cite{VW1}, \cite{VW2},
\begin{equation}\label{fde}
u_t=\La \phi_m(u),\quad u>0,
\end{equation}
where 
\begin{equation}\label{phi-m-defn}
\phi_m(u)=\left\{\begin{aligned}
&u^m/m\quad\mbox{ if }m\ne 0\\
&\log u\quad\mbox{ if }m=0\end{aligned}\right.
\end{equation} 
and the associated elliptic equation \cite{DKS}, \cite{Hs2}, \cite{Hs4}, \cite{Hu4}, 
\begin{equation}\label{elliptic-eqn}
\La \phi_m(v)+\alpha_m v+\beta x\cdot\nabla v=0, \quad v>0,
\end{equation}
where $\alpha_m$ and $\beta$ are some constants. Recently P.~Daskalopoulos, M.~del Pino, M. Fila, S.Y.~Hsu, K.M.~Hui, S.~Kim, J.~King, Ki-Ahm Lee, N. Sesum, M.~S\'aez, J. L. Vazquez, M. Winkler, E. Yanagida, E.~DiBenedetto, U.~Gianazza and N.~Liao, etc. have many results on 
\eqref{fde} and \eqref{elliptic-eqn}.
The equation \eqref{fde} appears in many physical models \cite{Ar}, \cite{DK}, \cite{V3} and in the study of Ricci and Yamabe flow on manifolds \cite{DS2}, \cite{H}, \cite{V2}, \cite{W}. When $m>1$, it appears in modelling the evolution of various diffusion processes such as the flow of a gas through a porous medium \cite{Ar}. When $m=1$, \eqref{fde} is the heat equation. When $0<m<1$, \eqref{fde} is the fast diffusion equation. When $n\geq 3$ and $g=u^{\frac{4}{n+2}}dx^2$ is a metric on $\R^n$ which evolves by the Yamabe flow
\begin{equation*}
\frac{\partial g}{\partial t}=-Rg\quad\mbox{ on }(0,T)
\end{equation*}
 where $R(\cdot,t)$ is the scalar curvature of the metric $g(\cdot,t)$, then $u$ satisfies \cite{DS2}, \cite{PS}, \cite{Y},
\begin{equation*}
u_t=\frac{n-1}{m}\Delta u^m\quad\mbox{ in }\R^n\times (0,T),\quad m=\frac{n-2}{n+2}
\end{equation*}
which after rescaling is equivalent to \eqref{fde}. Note that if  $n\ge 3$, $0\le m<\frac{n-2}{n}$, $\beta>0$, $\alpha_m=\frac{2\beta+\rho_1}{1-m}$ and $v^{(m)}$ is a solution of 
\eqref{elliptic-eqn} in $\R^n$ (or $\R^n\setminus\{0\}$), then with $\rho_1=1$ and $T>0$ the rescaled function
\begin{equation}\label{self-similar-soln1}
V^{(m)}(x,t)=\left(T-t\right)^{\alpha_m}v^{(m)}\left((T-t)^{\beta}x\right)
\end{equation}
is a self-similar solution of \eqref{fde} in $\R^n\times (0,T)$ ($(\R^n\setminus\{0\})\times (0,T)$, respectively) which vanishes at time $T$. Since solutions of \eqref{fde} which vanishes at a finite time usually behaves like self-similar solutions of the form \eqref{self-similar-soln1}, in order to understand the behaviour of the solutions of \eqref{fde}, it is important to study the properties of solutions of \eqref{elliptic-eqn}. 

For $m>\frac{(n-2)_+}{n}$, there are lots of studies on the solutions of \eqref{fde} (\cite{DK}, \cite{V3}). However there is not much study on the equations \eqref{fde} and \eqref{elliptic-eqn} for the case $n\ge 3$ and $0\le m<\frac{n-2}{n}$ until recently. This is because there is a big difference on the behaviour of solutions of \eqref{fde} for the case $\frac{(n-2)_+}{n}<m<1$ and the case $n\ge 3$, $0\le m<\frac{n-2}{n}$ \cite{DK}, \cite{HP}, \cite{V1}. For example for any $0\le u_0\in L_{loc}^1(\R^n)$, $u_0\not\equiv 0$, when $\frac{(n-2)_+}{n}<m<1$, there exists (\cite{HP}) a unique global positive smooth solution of \eqref{fde} in $\R^n\times (0,\infty)$ with initial data $u_0$ on $\R^n$. However for $n\ge 3$ and $0\le m<\frac{n-2}{n}$ the Barenblatt solutions \cite{DS1} 
\begin{equation*}
B_k(x,t)=\left(\frac{C_{\ast}}{k+(T-t)^{\frac{2}{n-2-nm}}|x|^2}\right)^{\frac{1}{1-m}}(T-t)^{\frac{n}{n-2-nm}},\qquad C_{\ast}=\frac{2(n-2-mn)}{1-m}, \quad k>0,
\end{equation*} 
satisfy \eqref{fde} in $\R^n\times (0,T)$ and vanishes identically at time $T$. 

For the subcritical case $m<\frac{(n-2)_+}{n}$,  M.~Fila and M.~Winkler [FW1--4] have obtained a lot of subtle phenomena for the solutions of \eqref{fde}. In \cite{FW1} and \cite{FW2} they proved the sharp rate of convergence of solutions of \eqref{fde} in $\R^n$ with $n>4$ and $0<m\le\frac{n-4}{n-2}$ to the Barenblatt solutions as the extinction time is approached. In \cite{FW3} they also proved the rate of convergence of solutions of \eqref{fde} in $\R^n$ to separable solutions of \eqref{fde} when $n>10$ and $0<m<\frac{(n-2)(n-10)}{(n-2)^2-4n+8\sqrt{n-1}}$. In \cite{FW4} they found an explicit dependence of the slow temporal growth rate of solutions of \eqref{fde} in $\R^n$ on the initial spatial growth rate.

Properties of singular solutions of \eqref{fde} are studied by E.~Chasseige, J.L.~Vazquez and M.~Winkler in the papers \cite{CV}, \cite{V4}, \cite{VW1} and \cite{VW2}. Existence of singular solution of \eqref{fde} for the case $\frac{n-2}{n}<m<1$ with initial value a nonnegative Borel measure on $\R^n$ which blows up at a singular set of $\R^n$ is proved by E.~Chasseige and J.L.~Vazquez in \cite{CV}. Finite blow-down or delay regularization behaviour for the solutions of the $2$-dimensional logarithmic diffusion equation 
\eqref{fde} (with $m=0$) was studied in \cite{V4}. Asymptotic oscillating behaviour of singular solutions of \eqref{fde}  in bounded domains of $\R^n$ with $0<m<\frac{n-2}{n}$ and $n\ge 3$  was studied in \cite{VW1} and the evolution of singularities of solutions of \eqref{fde} in bounded domains of $\R^n$ with $0<m<1$ and $n\ge 3$ was studied in \cite{VW2}. 

Another way to study the solutions of \eqref{fde} and \eqref{elliptic-eqn} is to study the singular limit of the solutions of \eqref{fde} and \eqref{elliptic-eqn} as $m\to 0$. Singular limit of solutions of \eqref{fde} in $\R^2\times (0,T)$ as $m\to 0^+$ and in $\Omega\times (0,\infty)$ for any bounded domain $\Omega\subset\R^n$, $n\ge 2$, as $m\to 0$ are proved by K.M.~Hui in \cite{Hu1} and \cite{Hu3}. Singular limit of solutions of \eqref{fde}  in $\R^n\times (0,\infty)$, $n\ge 2$, as $m\to 0^-$ is also proved by K.M.~Hui in \cite{Hu3}.
Singular limit of weak local solutions of \eqref{fde} in $O\times (0,\infty)$ as $m\to 0$ for any open set $O\subset\R^n$ is proved by E.~DiBenedetto, U.~Gianazza and N.~Liao in \cite{DGL}. For $n\ge 3$, $0<m\le\frac{n-2}{n}$  and either $\beta>0$ or $\alpha=0$, 
singular limit of solutions of 
\begin{equation*}
\La(v^m/m)+\alpha v+\beta x\cdot\nabla v=0, \quad v>0,\quad\mbox{ in }\R^n
\end{equation*} 
as $m\to 0^+$ is proved by S.Y.~Hsu in \cite{Hs2}.

In \cite{Hu4} K.M.~Hui proved for any $n\ge 3$, $0<m<\frac{n-2}{n}$, $\rho_1>0$, $\lambda>0$, $\beta\ge\beta_0^{(m)}$ and
\begin{equation}\label{alpha-m-defn}
\alpha_m=\frac{2\beta+\rho_1}{1-m}
\end{equation} 
where
\begin{equation}\label{beta-0-defn}
\beta_0^{(m)}=\frac{m\rho_1}{n-2-nm}
\end{equation}
there exists a radially symmetric solution $v:=v^{(m)}$ of \eqref{elliptic-eqn} in $\R^n\setminus\{0\}$ which satisfies
\begin{equation}\label{blow-up-rate-at-x=0}
\lim_{|x|\to 0}|x|^{\frac{\alpha_m}{\beta}}v(x)=\lambda^{-\frac{\rho_1}{(1-m)\beta}}.
\end{equation} 
In this paper we will prove that as $m\to 0^+$, the radially symmetric solution $v^{(m)}$ of \eqref{elliptic-eqn} in $\R^n\setminus\{0\}$ with $\beta>0$ and $\alpha_m$ given by \eqref{alpha-m-defn} converges uniformly in $C^2(K)$ for any compact subset $K$ of $\R^n\setminus\{0\}$ to the solution $v=v^{(0)}$ of 
\begin{equation}\label{elliptic-log-eqn}
\La \log v+\alpha v+\beta x\cdot\nabla v=0,\quad v>0, \quad  \mbox{ in }\R^n\setminus\{0\},
\end{equation}
which satisfies
\begin{equation}\label{log v-soln-x=0-rate}
\lim_{|x|\to 0}|x|^{\frac{\alpha}{\beta}}v^{(0)}(x)=\lambda^{-\frac{\rho_1}{\beta}}
\end{equation}  
where $\alpha=\alpha_0=2\beta+\rho_1$.
We will also prove that if $u^{(m)}$ is the solution of \eqref{fde} in $(\R^n\setminus\{0\})\times(0,T)$ with $\beta>0$ and $\alpha_m$ given by \eqref{alpha-m-defn} which blows up near $\{0\}\times(0,T)$ at the rate  $|x|^{-\alpha_m/\beta}$, then as $m\to 0^+$, $u^{(m)}$ converges uniformly in $C^{2,1}(K)$ for any compact subset $K$ of $(\R^n\setminus\{0\})\times(0,T)$ to the solution $u$ of 
\begin{equation}\label{log-diffusion-eqn}
u_t=\La\log u,\quad u>0,\quad\mbox{ in }(\R^n\bs\{0\})\times(0,T).
\end{equation}
For any $n\geq 3$, $0\le m<\frac{n-2}{n}$, $\rho_1>0$, $\beta>\beta_0^{(m)}$, $\alpha_m=\frac{2\beta+\rho_1}{1-m}$ and $\lambda>0$, we also prove the uniqueness of radially symmetric solution $v^{(m)}$ of \eqref{elliptic-eqn} in $\R^n\setminus\{0\}$  which satisfies
\eqref{blow-up-rate-at-x=0} and obtain higher order estimates of $v^{(m)}$ near the blow-up point $x=0$.   

Unless stated otherwise we will now assume that $n\ge 3$, $0\le m<\frac{n-2}{n}$, $\rho_1>0$, $\lambda>0$, $\beta\ge\beta_0^{(m)}$ and  $\phi_m$, $\alpha_m$, $\beta_0^{(m)}$, are given by \eqref{phi-m-defn}, \eqref{alpha-m-defn}  and \eqref{beta-0-defn} respectively and $v=v^{(m)}$ is a radially symmetric solution of\eqref{elliptic-eqn} in $\R^n\setminus\{0\}$ which satisfies \eqref{blow-up-rate-at-x=0} for the rest of the paper. We now recall a result of \cite{Hu4}.

\begin{thm}[Theorem 1.1 of \cite{Hu4}]\label{elliptic-eqn-existence-thm}
Let $n\geq 3$, $0<m<\frac{n-2}{n}$, $\rho_1>0$, $\lambda>0$ and $\beta\ge\beta_0^{(m)}$. Then there exists a radially symmetric solution $v=v^{(m)}$ of \eqref{elliptic-eqn} in $\R^n\bs\{0\}$ which satisfies \eqref{blow-up-rate-at-x=0} and
\begin{equation}\label{eq-condition-of-g-lambda-radially-symmetric-solution-0}
(v^{(m)})'(r)\leq 0 \qquad \forall r=|x|>0.
\end{equation}
\end{thm}

In this paper we will prove the following main results.

\begin{thm}\label{thm-singular-rate-near-zero-of-solutions-1}
Let $n\ge 3$, $0\le m<\frac{n-2}{n}$, $\rho_1>0$, $\lambda>0$, $\beta>\beta_0^{(m)}$ and  $\phi_m$, $\alpha_m$, $\beta_0^{(m)}$,
be given by \eqref{phi-m-defn}, \eqref{alpha-m-defn}  and \eqref{beta-0-defn} respectively and let $v=v^{(m)}$ be a radially symmetric solution of\eqref{elliptic-eqn} in $\R^n\setminus\{0\}$ which satisfies \eqref{blow-up-rate-at-x=0}. Let $\tilde{w}(r)=r^{\alpha_m/\beta}v(r)$, $\rho=r^{\rho_1/\beta}$ and $\overline{w}(\rho)=\tilde{w}(r)$. Then $\overline{w}$ can be extended to a function in  $C^2([0,\infty))$ by setting 
\begin{equation}\label{w-w'-w''-x=0}
\overline{w}(0)=\lambda^{-\frac{\rho_1}{(1-m)\beta}} ,\quad\overline{w}_{\rho}(0)=A_1\lambda^{-\frac{m\rho_1}{(1-m)\beta}} \qquad \mbox{and} \qquad \overline{w}_{\rho\rho}(0)=A_2\lambda^{-\frac{(2m-1)\rho_1}{(1-m)\beta}}
\end{equation}
where 
\begin{equation*}
A_1=\frac{a_3}{a_2},\qquad A_2=\frac{a_3(ma_3-a_1)}{a_2^2},
\end{equation*}
and 
\begin{equation}\label{a1-2-3-4-defn}
a_1=\frac{(n-2)\beta-2m\alpha_m+\rho_1}{\rho_1}, \qquad a_2=\frac{\beta^2}{\rho_1}, \qquad a_3=\frac{\alpha_m\beta(n-2)-m\alpha_m^2}{\rho_1^2}.
\end{equation} 
Hence
\begin{equation*}
\left\{\begin{aligned}
&\overline{w}(\rho)=\lambda^{-\frac{\rho_1}{(1-m)\beta}}+A_1\lambda^{-\frac{m\rho_1}{(1-m)\beta}}\rho+\frac{A_2}{2}\lambda^{-\frac{(2m-1)\rho_1}{(1-m)\beta}}\rho^2+o\left(\rho^2\right)\quad\mbox{ as }\rho\to 0^+\notag\\
&\overline{w}_{\rho}(\rho)=A_1\lambda^{-\frac{m\rho_1}{(1-m)\beta}}+A_2\lambda^{-\frac{(2m-1)\rho_1}{(1-m)\beta}}\rho+o\left(\rho\right)\qquad\qquad\qquad\mbox{ as }\rho\to 0^+
\end{aligned}\right.
\end{equation*}
or equivalently
\begin{equation}\label{vm-vm'-asymptotic-behaviour}
\left\{\begin{aligned}
&v^{(m)}(r)=r^{-\alpha_m/\beta}\left[\lambda^{-\frac{\rho_1}{(1-m)\beta}}+A_1\lambda^{-\frac{m\rho_1}{(1-m)\beta}}r^{\frac{\rho_1}{\beta}}+\frac{A_2}{2}\lambda^{-\frac{(2m-1)\rho_1}{(1-m)\beta}}r^{\frac{2\rho_1}{\beta}} +o\left(r^{\frac{2\rho_1}{\beta}} \right)\right]\quad\mbox{ as }r\to 0^+\\
&(v^{(m)})'(r)=r^{-(\alpha_m/\beta)-1}\left[-\frac{\alpha_m}{\beta}\lambda^{-\frac{\rho_1}{(1-m)\beta}}-\frac{(2\beta+m\rho_1)}{(1-m)\beta}A_1\lambda^{-\frac{m\rho_1}{(1-m)\beta}}r^{\frac{\rho_1}{\beta}}+o\left(r^{\frac{\rho_1}{\beta}} \right)\right]\quad\mbox{ as }r\to 0^+.
\end{aligned}\right.
\end{equation}
\end{thm}

\begin{thm}\label{uniqueness-thm}
Let $n\geq 3$, $0\leq m<\frac{n-2}{n}$, $\rho_1>0$, $\lambda>0$, $\beta>\beta_0^{(m)}$ and $\alpha_m$ be given by \eqref{alpha-m-defn}. Let $v_1, v_2$ be radially symmetric solutions of\eqref{elliptic-eqn} in $\R^n\setminus\{0\}$  which satisfies \eqref{blow-up-rate-at-x=0}. Then
\begin{equation}\label{eq-uniqueness-of-solution-g-1-and-g-2-w-r-t-initial-constant-lambda}
v_1(r)=v_2(r) \qquad \forall r>0.
\end{equation}
\end{thm}

\begin{thm}\label{thm-behaviour-of-v-m-at-infty}
Let $n\geq 3$, $\rho_1>0$, $\beta>0$ and $\alpha=2\beta+\rho_1$. Suppose that $v=v^{(0)}$ is a radially symmetric solution of \eqref{elliptic-log-eqn} in $\R^n\setminus B_1$. Then
\begin{equation}\label{eq-limit-of-r-2-u-r-at-infty-00}
\lim_{r\to\infty}r^2v(r)=\frac{2(n-2)}{\alpha-2\beta}.
\end{equation}
\end{thm}

\begin{thm}\label{elliptic-singular-limit-thm}
Let $n\geq 3$, $\rho_1>0$, $\lambda>0$, $\beta>0$ and $\alpha=2\beta+\rho_1$. Let $\2{m}_0\in\left(0,\frac{n-2}{n}\right)$ satisfy $\beta\geq \beta_0^{(\2{m}_0)}$. For any $0<m<\2{m}_0$, let $\alpha_m$ be given by \eqref{alpha-m-defn} and let $v^{(m)}$ be the unique radially symmetric solution of \eqref{elliptic-eqn} in $\R^n\setminus\{0\}$  which satisfies \eqref{blow-up-rate-at-x=0} given by Theorem \ref{elliptic-eqn-existence-thm} and Theorem \ref{uniqueness-thm}. Then as $m\to 0^+$, $v^{(m)}$ converges uniformly in $C^2(K)$ for any compact subset $K$ of $\R^n\setminus\{0\}$ to the unique radially symmetric solution $v$ of \eqref{elliptic-log-eqn} which satisfies \eqref{log v-soln-x=0-rate}.
\end{thm}

\begin{thm}\label{singular-log-diffusion-eqn-uniqueness-thm}
Let $n\geq 3$, $\beta>0$, $\lambda_1\geq \lambda_2>0$ and $\alpha=2\beta+1$. Suppose that $0\le u_{0,1}\le u_{0,2}\in L_{loc}^{\infty}(\R^n)$. If $u_1$, $u_2\in C((\R^n\bs\{0\})\times (0,T))\cap L^{\infty}_{loc}((\R^n\bs\{0\})\times [0,T))$ are subsolution and supersolution of 
\begin{equation}\label{eq-log-diffusion-equation-in-R-to-n-spacse-bs-zero-324r5}
u_t=\La\log u,\quad u>0,\quad \mbox{ in }(\R^n\bs\{0\})\times(0,T)
\end{equation}
which satisfies
\begin{equation*}
u_i(x,0)=u_{0,i}(x) \quad \mbox{ in }\R^n\quad\forall i=1,2
\end{equation*}
and
\begin{equation}\label{eq-trapping-solution-u-i-between-V-lambda-1-and-V-lambda-2}
V_{\lambda_1}(x,t)\leq u_i(x,t)\leq V_{\lambda_2}(x,t) \qquad \mbox{ in }\left(\R^n\bs\left\{0\right\}\right)\times(0,T)\quad\forall i=1,2
\end{equation}
where
\begin{equation*}
V_{\lambda_i}(x,t)=(T-t)^{\alpha}v_{\lambda_i}\left((T-t)^{\beta}|x|\right)\quad\forall i=1,2
\end{equation*}
and $v_{\lambda_i}$ is the radially symmetric solution of \eqref{elliptic-log-eqn} which satisfies \eqref{log v-soln-x=0-rate}  with $\lambda=\lambda_1$, $\lambda_2$, respectively, then
\begin{equation}\label{u1<u2-ineqn}
u_1\leq u_2 \qquad \mbox{in $\left(\R^n\bs\left\{0\right\}\right)\times(0,T)$}.
\end{equation}
Hence if $u_{0,1}=u_{0,2}$, then $u_1=u_2$ in $\left(\R^n\bs\left\{0\right\}\right)\times(0,T)$.
\end{thm}

\begin{thm}\label{parabolic-singular-limit-thm1}
Let $n\geq 3$, $0<\2{m}_0<\frac{n-2}{n}$, $\lambda_1>\lambda_2>0$, $\beta\ge\beta_0^{(\2{m}_0)}$, $\alpha=2\beta+1$ and $T>0$. For any $0<m<\overline{m}_0$, let $\alpha_m$ be given by \eqref{alpha-m-defn} with $\rho_1=1$ and
\begin{equation}\label{self-similar-soln-defn1}
V_{\lambda_i}^{(m)}(x,t)=(T-t)^{\alpha_m}v^{(m)}_{\lambda_i}\left((T-t)^{\beta}x\right)\quad \forall i=1,2
\end{equation}
where $v_{\lambda_i}^{(m)}$ is the radially symmetric solution of \eqref{elliptic-eqn} in $\R^n\bs\{0\}$ which satisfies \eqref{blow-up-rate-at-x=0} with $\lambda=\lambda_1$, $\lambda_2$, respectively. Let $\left\{u_{0,m}\right\}_{0<m<\2{m}_0}\subset L^{\infty}_{loc}\left(\R^n\bs\left\{0\right\}\right)$, $u_{0,m}\geq 0$ for all $0<m<\2{m}_0$, be a family of functions satisfying 
\begin{equation}\label{eq-compare-between-radially-symmetric-sols-and-initial-datas}
V_{\lambda_1}^{(m)}(x,0)\leq u_{0,m}(x)\leq V_{\lambda_2}^{(m)}(x,0) \quad \mbox{ in $\R^n\bs\left\{0\right\}$}
\end{equation}
and
\begin{equation*}
u_{0,m}\to u_0 \quad \mbox{in $L^1_{loc}\left(\R^n\bs\left\{0\right\}\right)$ as $m\to 0^+$}.
\end{equation*}
For any $0<m<\2{m}_0$, let $u^{(m)}$ be a solution of 
\begin{equation}\label{eq-cases-aligned-problem-of-parabolic-case-before-m-to-zero}
\begin{cases}
\begin{aligned}
&\quad u_t=\La(u^m/m),\quad u>0,\quad \mbox{ in }(\R^n\bs\{0\})\times(0,T)\\
&\quad u(x,0)=u_{0,m} \qquad \qquad \quad\mbox{ in }\R^n\bs\{0\}
\end{aligned}
\end{cases}
\end{equation}
given by Theorem 1.7 of \cite{Hu4} which satisfies
\begin{equation}\label{eq-compare-between-radially-symmetric-sols-and-sollutions}
V_{\lambda_1}^{(m)}(x,t)\leq u^{(m)}(x,t)\leq V_{\lambda_2}^{(m)}(x,t) \quad \mbox{ in }(\R^n\bs\{0\})\times(0,T).
\end{equation} 
Then $u^{(m)}$ converges uniformly in $C^{2+\theta,1+\frac{\theta}{2}}(K)$ for some constant $\theta\in (0,1)$ and any  compact subset $K$ of $\left(\R^n\bs\left\{0\right\}\right)\times(0,T)$ to the solution $u$ of 
\begin{equation}\label{eq-cases-aligned-problem-of-parabolic-case-after-m-to-zero}
\begin{cases}
\begin{aligned}
&u_t=\La\log u, u>0,\quad\mbox{ in }(\R^n\bs\{0\})\times (0,T)\\
&u(x,0)=u_0\qquad\qquad\mbox{ in }\R^n\bs\{0\}
\end{aligned}
\end{cases}
\end{equation}
as $m\to 0^+$ and $u$ satisfies
\begin{align}\label{eq-solution-u-trapped-by-V-sub-is}
V_1(x,t)\leq u(x,t)\leq V_2(x,t) \qquad  \mbox{in $\left(\R^n\bs\{0\}\right)\times(0,T)$} 
\end{align}
where
\begin{equation*}
V_i(x,t)=\left(T-t\right)^{\alpha}v_{\lambda_i}\left(\left(T-t\right)^{\beta}|x|\right)=\lim_{m\to 0}V_{\lambda_i}^{(m)}(x,t) \qquad \forall i=1,2
\end{equation*}
and $v_{\lambda_i}$ is the radially symmetric solution of \eqref{elliptic-log-eqn} given by Theorem \ref{elliptic-eqn-existence-thm} which satisfies \eqref{log v-soln-x=0-rate} with $\lambda=\lambda_i$, $i=1,2$, respectively.
\end{thm}

\begin{remark}
By Lemma 5.1 of \cite{Hu4} for any $n\ge 3$, $0<m<\frac{n-2}{n}$, and $\lambda_1>\lambda_2>0$, $T>0$, $\beta>\beta_0^m$,  $\alpha_m=\frac{2\beta+1}{1-m}$, $0\le u_0\in L_{loc}^{\infty}(\R^n\setminus\{0\})$, if $u_1$ and $u_2$ are two solutions of 
\begin{equation*}
\begin{cases}
\begin{aligned}
&\quad u_t=\La(u^m/m), u>0,\quad \mbox{ in }(\R^n\bs\{0\})\times(0,T)\\
&\quad u(x,0)=u_0 \qquad \qquad \quad\mbox{in }\R^n\bs\{0\}
\end{aligned}
\end{cases}
\end{equation*}
which satisfies \eqref{eq-compare-between-radially-symmetric-sols-and-sollutions}, then $u_1=u_2$ in 
$(\R^n\setminus\{0\})\times(0,T)$.
\end{remark}

The plan of the paper is as follows.  We will prove Theorem \ref{thm-singular-rate-near-zero-of-solutions-1} and Theorem \ref{uniqueness-thm} in section two.  We will prove Theorem \ref{thm-behaviour-of-v-m-at-infty} in section three and Theorem \ref{elliptic-singular-limit-thm}, Thneorem \ref{singular-log-diffusion-eqn-uniqueness-thm}, and  Theorem \ref{parabolic-singular-limit-thm1} in section four.

We start with some definitions. We say that $u$ is a solution of \eqref{fde} in $(\R^n\setminus\{0\})\times (0,T)$ if $u\in C^{2,1}((\R^n\setminus\{0\})\times (0,T))\cap L_{loc}^{\infty}((\R^n\setminus\{0\})\times (0,T))$ is positive in $(\R^n\setminus\{0\})\times (0,T)$ and satisfies \eqref{fde} in the classical sense in $(\R^n\setminus\{0\})\times (0,T)$. We say that $u$ is a subsolution (supersolution, respectively) of \eqref{fde} in $(\R^n\setminus\{0\})\times (0,T)$ if $u\in C((\R^n\setminus\{0\})\times (0,T))\cap L_{loc}^{\infty}((\R^n\setminus\{0\})\times (0,T))$ is positive in $(\R^n\setminus\{0\})\times (0,T)$ and satisfies 
\begin{equation}
\int_{\R^n}u(x,t_2)\eta(x,t_2)\,dx\le\int_{t_1}^{t_2}\int_{\R^n}(u\eta_t+\phi_m(u)\Delta\eta)\,dx\,dt+\int_{\R^n}u(x,t_1)\eta(x,t_1)\,dx\quad\forall T>t_2>t_1>0
\end{equation} 
($\ge$, respectively) for any $\eta\in C_0^{2,1}((\R^n\setminus\{0\})\times (0,T))$. For any $0\le u_0\in L_{loc}^{\infty}(\R^n\setminus\{0\})$, we say that a solution (or subsolution or supersolution) $u$ of \eqref{fde} in $(\R^n\setminus\{0\})\times (0,T)$ has initial value $u_0$ if $u(\cdot,t)\to u_0$ in $L_{loc}^1(\R^n\setminus\{0\})$ as $t\to 0$. 

We say that $v$ is a solution of \eqref{elliptic-eqn} in $\R^n\setminus\{0\}$ if $u\in C^{2,1}(\R^n\setminus\{0\})$ is positive in $\R^n\setminus\{0\}$ and satisfies \eqref{elliptic-eqn} in the classical sense in $R^n\setminus\{0\}$.  For any $R>0$, let $B_R=\{x\in\R^n:|x|<R\}$.

\section{uniqueness of radially symmetric solutions and higher order estimates at the origin}
\setcounter{equation}{0}
\setcounter{thm}{0}

In this section we will prove the uniqueness of radially symmetric solution $v^{(m)}$ of \eqref{elliptic-eqn} in $\R^n\setminus\{0\}$  which satisfies
\eqref{blow-up-rate-at-x=0} and obtain higher order estimates of $v^{(m)}$ near the blow-up point $x=0$. 

Let $\tilde{w}(r)=r^{\frac{\alpha_m}{\beta}}v^{(m)}(r)$, $\rho=r^{\frac{\rho_1}{\beta}}$ and 
$\overline{w}(\rho)=\tilde{w}(r)$. Then by the proof of Theorem 1.1 of \cite{Hu4}, $\tilde{w}$  satisfies
\begin{align}
&\left(\frac{\tilde{w}_r}{\tilde{w}}\right)_r+\frac{n-1-\frac{2m\alpha_m}{\beta}}{r}\cdot\frac{\tilde{w}_r}{\tilde{w}}+m\left(\frac{\tilde{w}_r}{\tilde{w}}\right)^2+\frac{\beta r^{-1-\frac{\rho_1}{\beta}}\tilde{w}_r}{\tilde{w}^m}=\frac{\alpha_m}{\beta}\cdot\frac{n-2-\frac{m\alpha_m}{\beta}}{r^2}\quad\forall r>0\notag\\
\Rightarrow\quad&\left(\frac{\overline{w}_{\rho}}{\overline{w}}\right)_{\rho}+m\left(\frac{\overline{w}_{\rho}}{\overline{w}}\right)^2+\frac{a_1}{\rho}\cdot\frac{\overline{w}_{\rho}}{\overline{w}}+\frac{a_2}{\rho^2}\cdot\frac{\overline{w}_{\rho}}{\overline{w}^m}=\frac{a_3}{\rho^2}\qquad\qquad\qquad\qquad\qquad\qquad\forall\rho>0
\label{eq-for-overline-q-corresponding-to-v-not-u}
\end{align}
where $a_1$, $a_2$ and $a_3$ are constants given by \eqref{a1-2-3-4-defn}. Note that by \eqref{blow-up-rate-at-x=0},
\begin{equation}\label{q-bar-value-x=0}
\lim_{\rho\to 0^+}\overline{w}(\rho)=\lambda^{-\frac{\rho_1}{(1-m)\beta}}.
\end{equation}
Hence $\2{w}(\rho)$ can be extended to a continuous function on $[0,\infty)$ by letting $\2{w}(0)=\lambda^{-\frac{\rho_1}{(1-m)\beta}}$. 

\begin{lemma}\label{lem-strictly-positivity-of-overline-q-rho-over-zeor-to-infty}
Let $n\ge 3$, $0\le m<\frac{n-2}{n}$, $\rho_1>0$, $\lambda>0$ and $\beta>\beta_0^{(m)}$. Then
\begin{equation}\label{eq-claim-positivity-of-w-=-overline-q-sub-rho-1}
\overline{w}_{\rho}(\rho)>0 \qquad \forall \rho>0
\end{equation}
or equivalently,
\begin{equation*}
v^{(m)}(r)+\frac{\beta}{\alpha_m}r(v^{(m)})'(r)>0 \qquad \forall r>0.
\end{equation*}
Hence
\begin{equation}\label{vm-lower-bd5}
v^{(m)}(r)\ge\lambda^{-\frac{\rho_1}{(1-m)\beta}}r^{-\frac{\alpha_m}{\beta}}\quad \forall r=|x|>0.
\end{equation}
\end{lemma}
\begin{proof}
Suppose that \eqref{eq-claim-positivity-of-w-=-overline-q-sub-rho-1} does not hold. Then there exists a constant $\rho_2>0$ such that
\begin{equation}\label{q-bar-derivative<0}
\overline{w}_{\rho}(\rho_2)\leq 0.
\end{equation}
Since $\beta>\beta_0^{(m)}$, $a_3>0$. Then by \eqref{eq-for-overline-q-corresponding-to-v-not-u} and \eqref{q-bar-derivative<0},
\begin{equation}\label{q-bar-2nd-derivative>0}
\left(\rho^{a_1}\cdot\overline{w}^m\cdot\frac{\overline{w}_{\rho}}{\overline{w}}\right)_{\rho}(\rho_2)=-a_2\rho_2^{a_1-2}\overline{w}_\rho(\rho_2)+a_3\rho_2^{a_1-2}\overline{w}(\rho_2)^m\geq a_3\rho_2^{a_1-2}\overline{w}(\rho_2)^m>0.
\end{equation}
Hence by \eqref{q-bar-derivative<0} and \eqref{q-bar-2nd-derivative>0} there exists a constant $b\in(0,\rho_2)$ such that
\begin{equation*}
\overline{w}_{\rho}(\rho)<0 \qquad \mbox{in $(\rho_2-b,\rho_2)$}.
\end{equation*}
Let $(\rho_3,\rho_2)$, $\rho_3\in [0,\rho_2)$,  be the maximal interval such that
\begin{equation}\label{eq-strictly-less-than-zero-of-overline-q-rho-34}
\overline{w}_{\rho}(\rho)<0 \qquad \forall \rho\in(\rho_3,\rho_2).
\end{equation}
If $a_1\geq 0$, by \eqref{eq-for-overline-q-corresponding-to-v-not-u} and \eqref{eq-strictly-less-than-zero-of-overline-q-rho-34},
\begin{align}
(\phi_m(\overline{w}))_{\rho\rho}(\rho)=&\left(\overline{w}^m\cdot\frac{\overline{w}_{\rho}}{\overline{w}}\right)_{\rho}(\rho)\ge a_3\frac{\overline{w}(\rho)^m}{\rho^2}\ge a_3\frac{\overline{w}(\rho_2)^m}{\rho^2}>0\quad\forall\rho\in(\rho_3,\rho_2)
\label{eq-inequality-of-overline-q-rho-over-overine-q-for-contradiction-3r1-1}\\
\Rightarrow\quad (\phi_m(\2{w}))_{\rho}(\rho_2)\ge&(\phi_m(\2{w}))_{\rho}(\rho)+a_3\2{w}^m(\rho_2)\left(\frac{1}{\rho}-\frac{1}{\rho_2}\right) \qquad\qquad\forall \rho\in(\rho_3,\rho_2)\notag\\
\Rightarrow \qquad\,\,\,\phi_m(\overline{w}(\rho))\ge&\phi_m(\2{w}(\rho_2))+\left(\phi_m(\2{w})_{\rho}(\rho_2)+a_3\rho_2^{-1}\overline{w}^m(\rho_2)\right)(\rho-\rho_2)+a_3\2{w}^m(\rho_2)\log \left(\frac{\rho_2}{\rho}\right)\quad\forall \rho_3<\rho<\rho_2 
\label{eq-inequality-of-overline-q-rho-over-overine-q-for-contradiction-3r1}.
\end{align}
If $a_1<0$, then by \eqref{eq-for-overline-q-corresponding-to-v-not-u} and \eqref{eq-strictly-less-than-zero-of-overline-q-rho-34},
\begin{align}
&(\rho^{a_1}(\phi_m(\overline{w}))_{\rho})_{\rho}=\left(\rho^{a_1}\cdot\overline{w}^m\cdot\frac{\overline{w}_{\rho}}{\overline{w}}\right)_{\rho}(\rho)\geq a_3\rho^{a_1-2}\overline{w}(\rho)^m\geq a_3\overline{w}(\rho_2)^m\rho^{a_1-2}>0 \qquad \forall\rho\in (\rho_3,\rho_2)
\label{eq-inequality-of-overline-q-rho-over-overine-q-for-contradiction-3r2-1}\\
&\qquad \Rightarrow\qquad\quad\rho_2^{a_1}(\phi_m(\overline{w}))_{\rho}(\rho_2)\ge\rho^{a_1}(\phi_m(\overline{w}))_{\rho}(\rho)+\frac{a_3\overline{w}(\rho_2)^{m}}{1-a_1}\left(\rho^{a_1-1}-\rho_2^{a_1-1}\right)\quad\forall \rho_3<\rho<\rho_2 \notag\\
&\qquad \Rightarrow\quad\,\,\rho_2^{a_1}(\phi_m(\overline{w}))_{\rho}(\rho_2)\rho^{-a_1}\ge(\phi_m(\overline{w}))_{\rho}(\rho)
+\frac{a_3\overline{w}(\rho_2)^{m}}{1-a_1}\left(\rho^{-1}-\rho_2^{a_1-1}\rho^{-a_1}\right)\quad\forall \rho_3<\rho<\rho_2 \notag\\
&\qquad \Rightarrow\qquad  \phi_m(\overline{w}(\rho))\ge\phi_m(\overline{w}(\rho_2))+C_1(\rho^{1-a_1}-\rho_2^{1-a_1})+\frac{a_3\overline{w}(\rho_2)^{m}}{1-a_1}\log (\rho_2/\rho)\quad\forall \rho_3<\rho<\rho_2 
\label{eq-inequality-of-overline-q-rho-over-overine-q-for-contradiction-3r2}.
\end{align}
where 
\begin{equation*}
C_1=\frac{1}{1-a_1}\left(\rho_2^{a_1}(\phi_m(\overline{w}))_{\rho}(\rho_2)+\frac{a_3\rho_2^{a_1-1}\overline{w}(\rho_2)^{m}}{1-a_1}\right).
\end{equation*}
If $\rho_3=0$, then by \eqref{eq-inequality-of-overline-q-rho-over-overine-q-for-contradiction-3r1} and \eqref{eq-inequality-of-overline-q-rho-over-overine-q-for-contradiction-3r2},
\begin{equation*}
\overline{w}(\rho)\to\infty \qquad \mbox{as $\rho\to 0^+$}
\end{equation*}
which contradicts \eqref{q-bar-value-x=0}. Hence $\rho_3>0$ and 
\begin{equation}\label{eq-condition-fo-overline-q-at-rho-3-equal-to-zeor}
\overline{w}_{\rho}(\rho_3)=0.
\end{equation} 
By \eqref{eq-inequality-of-overline-q-rho-over-overine-q-for-contradiction-3r1-1}, \eqref{eq-inequality-of-overline-q-rho-over-overine-q-for-contradiction-3r2-1} and \eqref{eq-condition-fo-overline-q-at-rho-3-equal-to-zeor}, 
\begin{equation*}
\overline{w}_{\rho}(\rho)>0\quad\forall\rho\in(\rho_3,\rho_2)
\end{equation*}
which contradicts \eqref{eq-strictly-less-than-zero-of-overline-q-rho-34}. Hence no such $\rho_2>0$ exists and \eqref{eq-claim-positivity-of-w-=-overline-q-sub-rho-1} follows. By \eqref{q-bar-value-x=0} and \eqref{eq-claim-positivity-of-w-=-overline-q-sub-rho-1},  \eqref{vm-lower-bd5} follows.
\end{proof}

\begin{lemma}\label{lem-limit-of-w-rho-1-over-overline-q--rgo}
Let $n\ge 3$, $0\le m<\frac{n-2}{n}$, $\rho_1>0$, $\lambda>0$ and $\beta>\beta_0^{(m)}$.  Then
\begin{equation}\label{eq-limit-of-w-rho-1-over-overline-q--rgo-to-a-2-aiver-a-3}
\lim_{\rho\to 0^+}\overline{w}_{\rho}(\rho)=\frac{a_3}{a_2}\lambda^{-\frac{m\rho_1}{(1-m)\beta}}=\frac{\alpha_m\beta(n-2)-m\alpha^2_m}{\beta^2\rho_1}\cdot\lambda^{-\frac{m\rho_1}{(1-m)\beta}}
\end{equation}
and
\begin{equation}\label{v-m-derivative-at-0}
\lim_{r\to 0^+}r^{\frac{\alpha_m}{\beta}+1}(v^{(m)})'(r)=-\frac{\alpha_m}{\beta}\lambda^{-\frac{\rho_1}{(1-m)\beta}}
\end{equation}
where $a_1$, $a_2$ and $a_3$ are constants given by \eqref{a1-2-3-4-defn}.
Hence $\overline{w}$ can be extended to a function in $C^1([0,\infty))$ by letting $\overline{w}(0)=\lambda^{-\frac{\rho_1}{(1-m)\beta}}$ and $\overline{w}_{\rho}(0)=\frac{a_3}{a_2}\lambda^{-\frac{m\rho_1}{(1-m)\beta}}$. 
\end{lemma}
\begin{proof}
Let 
\begin{equation*}
q(\rho)=\frac{1}{\overline{w}_{\rho}(\rho)}.
\end{equation*}
By \eqref{eq-for-overline-q-corresponding-to-v-not-u} $q(\rho)$ satisfies
\begin{equation}\label{eq-eq-for-w-with-respect-to-v}
q_{\rho}(\rho)=-\frac{(1-m)}{\overline{w}(\rho)}+\frac{q(\rho)}{\rho}\left[a_1+\frac{\overline{w}(\rho)}{\rho}\left(\frac{a_2}{\overline{w}(\rho)^m}-a_3q(\rho)\right)\right] \quad \forall \rho>0.
\end{equation}
By Lemma \ref{lem-strictly-positivity-of-overline-q-rho-over-zeor-to-infty},
\begin{equation}\label{eq-positivity-of-w-with-respect-to-v-for-allo-rho}
q(\rho)>0 \quad \forall \rho>0.
\end{equation}
By \eqref{q-bar-value-x=0} and \eqref{eq-claim-positivity-of-w-=-overline-q-sub-rho-1} there exists a constant $\rho_2>0$ such that
\begin{equation}\label{eq-lower-bound-of-overline-q-in-the-nbd-of-zero}
\lambda^{-\frac{\rho_1}{(1-m)\beta}}<\overline{w}(\rho)<2\lambda^{-\frac{\rho_1}{(1-m)\beta}} \qquad \forall 0<\rho<\rho_2.
\end{equation}
We now claim that there exists a constant $\rho_0>0$ such that 
\begin{equation}\label{eq-claim-bounds-of-w-rho-upper-and-lower}
 \frac{a_2}{8a_3}\lambda^{\frac{m\rho_1}{(1-m)\beta}}\leq q(\rho)\leq \frac{2^{1+m}a_2}{a_3}\lambda^{\frac{m\rho_1}{(1-m)\beta}} \qquad \forall 0<\rho<\rho_0.
\end{equation}
\indent To prove the inequality on the right hand side of \eqref{eq-claim-bounds-of-w-rho-upper-and-lower}, we first suppose that the inequality on the right hand side of \eqref{eq-claim-bounds-of-w-rho-upper-and-lower} does not hold for any $\rho_0>0$. Then there exists a constant 
\begin{equation}\label{eq-assumption-of-w-above-some-boundes}
0<\rho_3<\min\left\{\rho_2, \frac{a_2\lambda^{-\frac{\rho_1}{\beta}}}{8|a_1|+1}\right\}
\end{equation} 
such that
\begin{equation*}\label{eq-assumption-of-w-above-some-boundes-value}
q(\rho_3)>\frac{2^{1+m}a_2}{a_3}\lambda^{\frac{m\rho_1}{(1-m)\beta}}.
\end{equation*}
Then by continuity of $q(\rho)$ on $(0,\infty)$, there exists a maximal interval $(\rho_4,\rho_3)$, $(0\leq\rho_4<\rho_3)$ such that 
\begin{equation}\label{eq-condition-of-w-on-the-maximal-intervals}
q(\rho)>\frac{2^{1+m}a_2}{a_3}\lambda^{\frac{m\rho_1}{(1-m)\beta}} \qquad \forall \rho\in (\rho_4,\rho_3). 
\end{equation}
By \eqref{eq-eq-for-w-with-respect-to-v}, \eqref{eq-positivity-of-w-with-respect-to-v-for-allo-rho}, \eqref{eq-lower-bound-of-overline-q-in-the-nbd-of-zero}, \eqref{eq-assumption-of-w-above-some-boundes} and \eqref{eq-condition-of-w-on-the-maximal-intervals}, $q(\rho)$ satisfies

\begin{align}\label{eq-aligned-inequality-for-w-rho}
q_{\rho}(\rho)&\leq \frac{q(\rho)}{\rho}\left[\left(a_1-\frac{a_3\overline{w}(\rho)}{4\rho}q(\rho)\right)+\frac{\overline{w}(\rho)}{\rho}\left(\frac{a_2}{\overline{w}(\rho)^m}-\frac{a_3}{2}q(\rho)\right)-\frac{a_3\overline{w}(\rho)}{4\rho}q(\rho)\right]\notag\\
&\le\frac{q(\rho)}{\rho}\left[\left(a_1-\frac{a_2\lambda^{-\frac{\rho_1}{\beta}}}{4\rho}\right)+\frac{\overline{w}(\rho)}{\rho}\left( a_2\lambda^{\frac{m\rho_1}{(1-m)\beta}}-\frac{a_3}{2}q(\rho)\right)-\frac{a_3\lambda^{-\frac{\rho_1}{(1-m)\beta}}}{4\rho}q(\rho)\right]\\
& \leq -\frac{a_3\lambda^{-\frac{\rho_1}{(1-m)\beta}}}{4\rho^2}q(\rho)^2<0
\end{align}
in $(\rho_4,\rho_3)$. Dividing \eqref{eq-aligned-inequality-for-w-rho} by $q(\rho)^2$ and integrating over $(\rho,\rho_3)$, $\rho_4<\rho<\rho_3$, 
\begin{align}
&\overline{w}_{\rho}(\rho)=\frac{1}{q(\rho)}\leq\left(\frac{1}{q(\rho_3)}+\frac{a_3\lambda^{-\frac{\rho_1}{(1-m)\beta}}}{4\rho_3}\right)-\frac{a_3\lambda^{-\frac{\rho_1}{(1-m)\beta}}}{4\rho}\quad \forall \rho\in(\rho_4,\rho_3)\notag\\
\Rightarrow\quad&\overline{w}(\rho)\geq \overline{w}(\rho_3)+\left(\frac{1}{q(\rho_3)}+\frac{a_3\lambda^{-\frac{\rho_1}{(1-m)\beta}}}{4\rho_3}\right)(\rho-\rho_3)+\frac{a_3\lambda^{-\frac{\rho_1}{(1-m)\beta}}}{4}\log(\rho_3/\rho) \quad \forall \rho\in(\rho_4,\rho_3).\label{eq-inequality-for-overline-q-on-rho-4-amnd-rho-3}
\end{align}
If $\rho_4=0$, then by \eqref{eq-inequality-for-overline-q-on-rho-4-amnd-rho-3},
\begin{equation*}
\lim_{\rho\to0^+}\overline{w}(\rho)=\infty
\end{equation*}
which contradicts \eqref{q-bar-value-x=0}. Hence $\rho_4>0$ and
\begin{equation}\label{w-lhs-value}
q(\rho_4)=\frac{2^{1+m}a_2}{a_3}\lambda^{\frac{m\rho_1}{(1-m)\beta}}.
\end{equation}
By \eqref{eq-aligned-inequality-for-w-rho} and \eqref{w-lhs-value},
\begin{equation*}
q(\rho)<q(\rho_4)=\frac{2^{1+m}a_2}{a_3}\lambda^{\frac{m\rho_1}{(1-m)\beta}} \quad \forall \rho_4<\rho<\rho_3
\end{equation*}
which contradicts \eqref{eq-condition-of-w-on-the-maximal-intervals}. Hence no such $\rho_3>0$ exists and
\begin{equation}\label{eq-completed-of-upper-bound-of-w-rho}
q(\rho)\leq \frac{2^{1+m}a_2}{a_3}\lambda^{\frac{m\rho_1}{(1-m)\beta}} \qquad \forall 0<\rho<\min\left(\rho_2,\frac{a_2\lambda^{-\frac{\rho_1}{\beta}}}{8|a_1|+1}\right).
\end{equation}
Now suppose the first inequality of \eqref{eq-claim-bounds-of-w-rho-upper-and-lower} does not hold for any $\rho_0>0$. Then there exists a constant 
\begin{equation}\label{eq-assumption-of-w-lower-some-boundes}
0<\rho_5<\min\left\{\rho_2,\frac{a_2\lambda^{-\frac{\rho_1}{\beta}}}{8|a_1|+1}\right\}
\end{equation} 
such that
\begin{equation}\label{eq-assumption-of-w-lower-some-boundes-value}
q(\rho_5)<\frac{a_2}{8a_3}\lambda^{\frac{m\rho_1}{(1-m)\beta}}.
\end{equation}
By \eqref{eq-lower-bound-of-overline-q-in-the-nbd-of-zero} and \eqref{eq-assumption-of-w-lower-some-boundes-value},
\begin{equation}\label{bar-q-w-upper-bd}
\overline{w}(\rho_5)q(\rho_5)<\frac{a_2\lambda^{-\frac{\rho_1}{\beta}}}{4a_3}.
\end{equation}
Then by \eqref{bar-q-w-upper-bd} and continuity of $\overline{w}(\rho)q(\rho)$ on $(0,\infty)$ there exists a maximal interval $(\rho_6,\rho_5)$ ($0\leq\rho_6<\rho_5$) such that 
\begin{equation}\label{eq-condition-of-w-on-the-maximal-intervals-2}
\overline{w}(\rho)q(\rho)<\frac{a_2\lambda^{-\frac{\rho_1}{\beta}}}{4a_3}\qquad \forall \rho\in (\rho_6,\rho_5). 
\end{equation}
By \eqref{eq-eq-for-w-with-respect-to-v}, \eqref{eq-positivity-of-w-with-respect-to-v-for-allo-rho}, \eqref{eq-lower-bound-of-overline-q-in-the-nbd-of-zero}, \eqref{eq-assumption-of-w-lower-some-boundes} and \eqref{eq-condition-of-w-on-the-maximal-intervals-2},
\begin{align}
\left(\overline{w}(\rho)q(\rho)\right)_{\rho}=&\overline{w}(\rho)q_{\rho}(\rho)+1\notag\\
\ge&\frac{\overline{w}(\rho)q(\rho)}{\rho}\left[\left(a_1+\frac{a_2\overline{w}(\rho)^{1-m}}{4\rho}\right)
+\frac{3}{4\rho}\left(a_2\overline{w}(\rho)^{1-m}-2a_3q(\rho)\overline{w}(\rho)\right)
+\frac{a_3}{2\rho}\overline{w}(\rho)q(\rho)\right]\notag\\
\ge&\frac{\overline{w}(\rho)q(\rho)}{\rho}\left[\left(a_1+\frac{a_2\lambda^{-\frac{\rho_1}{\beta}}}{\rho}\right)
+\frac{3}{4\rho}\left(a_2\lambda^{-\frac{\rho_1}{\beta}}-\frac{a_2\lambda^{-\frac{\rho_1}{\beta}}}{2}\right)
+\frac{a_3}{2\rho}\overline{w}(\rho)q(\rho)\right]\notag\\
\ge&\frac{a_3}{2\rho^2}\left(\overline{w}(\rho)q(\rho)\right)^2\qquad \mbox{on $(\rho_6,\rho_5)$}.\label{eq-aligned-inequality-for-w-rho-2}
\end{align}
Dividing \eqref{eq-aligned-inequality-for-w-rho-2} by $(\overline{w}(\rho)q(\rho))^2$ and integrating over $(\rho,\rho_5)$, $\rho_6<\rho<\rho_5$, we get
\begin{align}
&\frac{\overline{w}_{\rho}(\rho)}{\overline{w}(\rho)}=\frac{1}{\overline{w}(\rho)q(\rho)}\geq\left(\frac{\overline{w}_{\rho}(\rho_5)}{\overline{w}(\rho_5)}-\frac{a_3}{2\rho_5}\right)+\frac{a_3}{2\rho}\quad \forall \rho\in(\rho_6,\rho_5)\notag\\
\Rightarrow\quad&\log\overline{w}(\rho)\leq \log\overline{w}(\rho_5)+\left(\frac{\overline{w}_{\rho}(\rho_5)}{\overline{w}(\rho_5)}-\frac{a_3}{2\rho_5}\right)(\rho-\rho_5)+\frac{a_3}{2}\log\left(\frac{\rho}{\rho_5}\right) \quad \forall \rho\in(\rho_6,\rho_5).
\label{eq-inequality-for-overline-q-on-rho-4-amnd-rho-3-2}
\end{align}
If $\rho_6=0$, then by \eqref{eq-inequality-for-overline-q-on-rho-4-amnd-rho-3-2},
\begin{equation*}
\lim_{\rho\to0^+}\overline{w}(\rho)=0
\end{equation*}
which contradicts \eqref{q-bar-value-x=0}. Hence $\rho_6>0$ and
\begin{equation}\label{eq-existence-of-rho-6-be-positive-and-value-a-2-over-2-a-3}
\overline{w}(\rho_6)q(\rho_6)=\frac{a_2\lambda^{-\frac{\rho_1}{\beta}}}{4a_3}.
\end{equation}
By \eqref{eq-aligned-inequality-for-w-rho-2} and \eqref{eq-existence-of-rho-6-be-positive-and-value-a-2-over-2-a-3}, 
\begin{equation*}
\overline{w}(\rho)q(\rho)>\frac{a_2\lambda^{-\frac{\rho_1}{\beta}}}{4a_3} \quad \forall \rho_6<\rho<\rho_5
\end{equation*} 
which contradicts \eqref{eq-condition-of-w-on-the-maximal-intervals-2}. Hence no such $\rho_5>0$ exists and
\begin{equation}\label{eq-completed-of-lower-bound-of-w-rho}
q(\rho)\geq \frac{a_2}{8a_3}\lambda^{\frac{m\rho_1}{(1-m)\beta}}\qquad \forall 0<\rho<\min\left\{\rho_2,\frac{a_2\lambda^{-\frac{\rho_1}{\beta}}}{8|a_1|+1}\right\}.
\end{equation}
By \eqref{eq-completed-of-upper-bound-of-w-rho} and \eqref{eq-completed-of-lower-bound-of-w-rho}, \eqref{eq-claim-bounds-of-w-rho-upper-and-lower} holds for 
\begin{equation*}
\rho_0=\min\left\{\rho_2, \frac{a_2\lambda^{-\frac{\rho_1}{\beta}}}{8|a_1|+1}\right\}.
\end{equation*}
\indent Let $\left\{\rho_i\right\}\subset\R^+$ be a sequence such that $\rho_i\to 0$ as $i\to\infty$. Then, by \eqref{eq-claim-bounds-of-w-rho-upper-and-lower}, the sequence $\left\{\rho_i\right\}$ has a subsequence which we may assume without loss of generality to be the sequence $\left\{\rho_i\right\}$ itself such that
\begin{equation*}
q_{\infty}:=\lim_{i\to \infty}q(\rho_i) \quad \mbox{ exists}
\end{equation*} 
and 
\begin{equation}\label{eq-existence-of-w-rho-at-zero-+}
q_{\infty}\in\left[\frac{a_2}{8a_3}\lambda^{\frac{m\rho_1}{(1-m)\beta}},\frac{2^{1+m}a_2}{a_3}\lambda^{\frac{m\rho_1}{(1-m)\beta}}\right].
\end{equation}
By \eqref{eq-eq-for-w-with-respect-to-v},
\begin{align}
&\left(\rho^{-a_1}e^{a_2\int_{\rho}^{1}s^{-2}\overline{w}(s)^{1-m}\,ds}q(\rho)\right)_{\rho}=-\rho^{-a_1}\left[(1-m)\overline{w}(\rho)^{-1}+a_3\rho^{-2}\overline{w}(\rho)q(\rho)^2\right]e^{a_2\int_{\rho}^{1}s^{-2}\overline{w}(s)^{1-m}\,ds}\quad\forall \rho>0 \notag\\
\Rightarrow\quad&q(\rho)=\frac{q(1)+\int_{\rho}^1s^{-a_1}\left[(1-m)\overline{w}(s)^{-1}+a_3s^{-2}\overline{w}(s)q(s)^2\right]e^{a_2\int_s^1\sigma^{-2}\overline{w}(\sigma)^{1-m}\,d\sigma}\,ds}{\rho^{-a_1}e^{a_2\int_{\rho}^{1}s^{-2}\overline{w}(s)^{1-m}\,ds}}\quad\forall \rho>0.\label{eq-aligned-expression-of-w-with-quotient-and-integartion}
\end{align}
Since $\lim_{s\to 0^+}s^le^{\frac{1}{s}}=\infty$ for any $l\in\R$, by \eqref{q-bar-value-x=0} and \eqref{eq-claim-bounds-of-w-rho-upper-and-lower},
\begin{equation*}
\int_{\rho_i}^1s^{-a_1}\left[(1-m)\overline{w}(s)^{-1}+a_3s^{-2}\overline{w}(s)q(s)^2\right]e^{a_2\int_s^1\sigma^{-2}\overline{w}(\sigma)^{1-m}\,d\sigma}\,ds\to\infty \qquad \mbox{as $i\to\infty$}
\end{equation*}
and
\begin{equation*}
\rho_i^{-a_1}e^{a_2\int_{\rho_i}^{1}s^{-2}\overline{w}(s)^{1-m}\,ds}\to\infty \qquad \mbox{as $i\to \infty$}.
\end{equation*}
Hence by \eqref{q-bar-value-x=0}, \eqref{eq-existence-of-w-rho-at-zero-+}, \eqref{eq-aligned-expression-of-w-with-quotient-and-integartion} and l'Hospital rule,
\begin{align}
q_{\infty}=&\lim_{i\to \infty}q(\rho_i)=\lim_{i\to \infty}\frac{-\rho_i^{-a_1}\left[(1-m)\overline{w}(\rho_i)^{-1}+a_3\rho_i^{-2}\overline{w}(\rho_i)q(\rho_i)^2\right]e^{a_2\int_{\rho_i}^{1}s^{-2}\overline{w}(s)^{1-m}\,ds}}{-a_1\rho_i^{-(a_1+1)}e^{a_2\int_{\rho_i}^{1}s^{-2}\overline{w}(s)^{1-m}\,ds}-a_2\rho_i^{-(a_1+2)}\overline{w}(\rho_i)^{1-m}e^{a_2\int_{\rho_i}^{1}s^{-2}\overline{w}(s)^{1-m}\,ds}}\notag\\
=&\lim_{i\to\infty}\frac{(1-m)\rho_i^2\overline{w}(\rho_i)^{-1}+a_3\overline{w}(\rho_i)q(\rho_i)^2}{a_1\rho_i+a_2\overline{w}(\rho_i)^{1-m}}
=\frac{a_3}{a_2}\lambda^{-\frac{m\rho_1}{(1-m)\beta}}q_{\infty}^2.\label{eq-after-lhospital-rule-to-find-w-at-zero}
\end{align}
Hence by \eqref{eq-existence-of-w-rho-at-zero-+} and \eqref{eq-after-lhospital-rule-to-find-w-at-zero},
\begin{equation*}
q_{\infty}=\frac{a_2}{a_3}\lambda^{\frac{m\rho_1}{(1-m)\beta}}.
\end{equation*}
Since the sequence $\left\{\rho_i\right\}$ is arbitrary, 
\begin{equation*}
\lim_{\rho\to 0^+}q(\rho)=\frac{a_2}{a_3}\lambda^{\frac{m\rho_1}{(1-m)\beta}}
\end{equation*}
and \eqref{eq-limit-of-w-rho-1-over-overline-q--rgo-to-a-2-aiver-a-3} follows. Since
\begin{equation}\label{v-m-q-bar-derivative-relation}
r^{\frac{\alpha_m}{\beta}+1}(v^{(m)})'(r)=\frac{\rho_1}{\beta}\rho\overline{w}_{\rho}(\rho)-\frac{\alpha_m}{\beta}r^{\frac{\alpha_m}{\beta}}v^{(m)}(r)\quad \forall \rho=r^{\frac{\rho_1}{\beta}}>0,
\end{equation}
by \eqref{blow-up-rate-at-x=0} and \eqref{eq-limit-of-w-rho-1-over-overline-q--rgo-to-a-2-aiver-a-3},
\begin{equation*}\label{eq-aligned-limit-of-g-prime-and-r-to-alpah-bover-beta-1-at-zero}
\begin{aligned}
\lim_{r\to 0^+}r^{\frac{\alpha_m}{\beta}+1}(v^{(m)})'(r)=\frac{\rho_1}{\beta}\lim_{\rho\to 0^+}\rho\overline{w}_{\rho}(\rho)-\frac{\alpha_m}{\beta}\lim_{r\to 0^+}r^{\frac{\alpha_m}{\beta}}v^{(m)}(r)=-\frac{\alpha_m}{\beta}\lambda^{-\frac{\rho_1}{(1-m)\beta}}
\end{aligned}
\end{equation*}
and \eqref{v-m-derivative-at-0} follows.
\end{proof}

\begin{lemma}\label{lem-limit-of-overline-q-rho-rho-1-over-overline-q--rgo}
Let $n\ge 3$, $0\le m<\frac{n-2}{n}$, $\rho_1>0$, $\lambda>0$ and $\beta>\beta_0^{(m)}$.  Then
\begin{equation}\label{eq-limit-of-overline-q-rho-rho-1-over-overline-q--rgo-to-a-2-aiver-a-3}
\lim_{\rho\to 0^+}\overline{w}_{\rho\rho}(\rho)=\frac{a_3(ma_3-a_1)}{a_2^2}\lambda^{-\frac{(2m-1)\rho_1}{(1-m)\beta}}
\end{equation}
where $a_1$, $a_2$ and $a_3$ are constants given by \eqref{a1-2-3-4-defn}.
Hence $\overline{w}$ can be extended to a function in $C^2([0,\infty))$ by defining $\overline{w}_{\rho}(0)$, $\overline{w}_{\rho}(0)$ and $\overline{w}_{\rho\rho}(0)$ by \eqref{w-w'-w''-x=0}.
\end{lemma}
\begin{proof}
Let $\tilde{v}(\rho)=\overline{w}_{\rho}(\rho)$. Then by \eqref{eq-for-overline-q-corresponding-to-v-not-u},
\begin{equation}\label{eq-aligned-invert-eq-for-overine=q-to-v}
\begin{aligned}
\tilde{v}'&=(1-m)\frac{\tilde{v}^2}{\overline{w}}-\frac{a_1}{\rho}\tilde{v}-\frac{a_2}{\rho^2}\overline{w}^{1-m}\tilde{v}+\frac{a_3}{\rho^2}\overline{w}\\
&=(1-m)\frac{\tilde{v}^2}{\overline{w}}-\frac{a_1}{\rho}\tilde{v}-\frac{a_2\overline{w}^{1-m}}{\rho^2}\left(\tilde{v}-\frac{a_3}{a_2}\cdot\overline{w}^m\right)\qquad\forall\rho>0.
\end{aligned}
\end{equation}
Let $v_1(\rho)$ be given by
\begin{equation}\label{eq-split-of-v-by-v-sub-zero-and-v-sub-1}
\tilde{v}(\rho)=v_0+v_1(\rho)\rho 
\end{equation}
where
\begin{equation*}
v_0=\frac{a_3}{a_2}\overline{w}^m(0).
\end{equation*}
Then by \eqref{eq-aligned-invert-eq-for-overine=q-to-v} for any $\rho>0$,
\begin{equation}\label{eq-aligned-eq-for-v-sub-1-before-integrals}
\begin{aligned}
&v_1'(\rho)\rho+v_1(\rho)\\
&\qquad =\tilde{v}'(\rho)=\frac{(1-m)}{\overline{w}(\rho)}\tilde{v}^2(\rho)-a_1v_1(\rho)-\frac{a_1v_0}{\rho}-\frac{a_2\overline{w}(\rho)^{1-m}}{\rho^2}\left[\frac{a_3}{a_2}\left( \overline{w}^m(0)  -\overline{w}^m(\rho)\right)+v_1(\rho)\rho\right].
\end{aligned}
\end{equation}
By the mean value theorem, for any $\rho>0$ there exists a constant $\xi=\xi(\rho)\in(0,\rho)$ such that
\begin{equation}\label{eq-mean-value-theorem-for-overline-q-to-m}
\overline{w}^m(\rho)-\overline{w}^m(0)=m\overline{w}(\xi)^{m-1}\overline{w}_{\rho}(\xi)\rho.
\end{equation}
By \eqref{eq-aligned-eq-for-v-sub-1-before-integrals} and \eqref{eq-mean-value-theorem-for-overline-q-to-m},
\begin{align}
v_1'(\rho)=&\frac{1}{\rho}\left[\frac{(1-m)}{\overline{w}(\rho)}\tilde{v}^2(\rho)-(1+a_1)v_1(\rho)\right]+\frac{a_2\overline{w}(\rho)^{1-m}}{\rho^2}\left[\frac{ma_3}{a_2}\overline{w}(\xi)^{m-1}\overline{w}_{\rho}(\xi)-v_1(\rho)-\frac{a_1v_0}{a_2\overline{w}(\rho)^{1-m}}\right]\nonumber\\
=&\frac{a_2\overline{w}(\rho)^{1-m}}{\rho}\left[\frac{1-m}{a_2\overline{w}(\rho)^{2-m}}\tilde{v}^2(\rho)-\frac{f_1(\rho)}{\rho}\right] \label{eq-eq-for-v-1-with-f-2-and-some-terms}
\end{align}
where
\begin{equation}\label{f1-defn}
f_1(\rho)=\frac{(1+a_1)}{a_2}\rho v_1(\rho)\overline{w}(\rho)^{m-1}+v_1(\rho)-f_2(\rho)
\end{equation}
and
\begin{equation*}
f_2(\rho)=\frac{ma_3}{a_2}\overline{w}(\xi(\rho))^{m-1}\overline{w}_{\rho}(\xi(\rho))-\frac{a_1a_3}{a_2^2}\overline{w}(\rho)^{m-1}\lambda^{-\frac{m\rho_1}{(1-m)\beta}}.
\end{equation*}
Let
$$
a_4=\frac{a_3(ma_3-a_1)}{a_2^2}.
$$
Without loss of generality we may assume that $a_4>0$. 
Then by \eqref{q-bar-value-x=0} and \eqref{eq-limit-of-w-rho-1-over-overline-q--rgo-to-a-2-aiver-a-3},
\begin{equation}\label{f1-x=0}
\lim_{\rho\to 0^+}f_2(\rho)=a_4\lambda^{-\frac{(2m-1)\rho_1}{(1-m)\beta}}
\end{equation}
and
\begin{equation}\label{rho-v1-x=0}
\lim_{\rho\to 0^+}\rho v_1(\rho)=0.
\end{equation}
Let $0<\3<1/5$. By \eqref{f1-x=0} and \eqref{rho-v1-x=0} there exists a constant $\rho_2>0$ such that
\begin{equation}\label{f1-upper-lower-bds}
(1-\3)a_4\lambda^{-\frac{(2m-1)\rho_1}{(1-m)\beta}}\le f_2(\rho)\le (1+\3)a_4\lambda^{-\frac{(2m-1)\rho_1}{(1-m)\beta}}\quad\forall 0<\rho\le\rho_2
\end{equation}
and
\begin{equation}\label{rho-nv1-upper-lower-bds}
\frac{(1+a_1)}{a_2}\rho |v_1(\rho)|\overline{w}(\rho)^{m-1}\le \3 a_4\lambda^{-\frac{(2m-1)\rho_1}{(1-m)\beta}}\quad\forall 0<\rho\le\rho_2.
\end{equation}
and \eqref{eq-lower-bound-of-overline-q-in-the-nbd-of-zero} hold. Let
\begin{equation*}
\rho_{\3}=\min \left(\rho_2,\frac{\3a_2a_4\lambda^{-\frac{(2m-1)\rho_1}{(1-m)\beta}}}{16(1-m)}\cdot\inf_{0<\rho\le 1}\frac{\overline{w}(\rho)^{2-m}}{\overline{w}_{\rho}(\rho)^2}\right)
\end{equation*} 
We claim that 
\begin{equation}\label{v1-upper-lower-bd}
(1-3\3)a_4\lambda^{-\frac{(2m-1)\rho_1}{(1-m)\beta}}
\le v_1(\rho)\le (1+3\3)a_4\lambda^{-\frac{(2m-1)\rho_1}{(1-m)\beta}}\quad\forall 0<\rho<\rho_{\3}.
\end{equation}
Suppose that the second inequality in \eqref{v1-upper-lower-bd} does not hold. Then there exists a constant $\rho_1'\in(0,\rho_{\3})$ such that
\begin{equation*}
v_1(\rho_1')>(1+3\3)a_4\lambda^{-\frac{(2m-1)\rho_1}{(1-m)\beta}}.
\end{equation*}
By continuity of $v_1(\rho)$ on $(0,\infty)$, there exists a maximal interval $(\rho_3,\rho_4)$ containing $\rho_1'$, $0\leq\rho_3<\rho_1'<\rho_4\leq\rho_{\3}$, such that
\begin{equation}\label{v1-lower-bd}
v_1(\rho)>(1+3\3)a_4\lambda^{-\frac{(2m-1)\rho_1}{(1-m)\beta}} \quad \forall\rho\in(\rho_3,\rho_4).
\end{equation}
Then by \eqref{f1-defn}, \eqref{f1-upper-lower-bds}, \eqref{rho-nv1-upper-lower-bds} and \eqref{v1-lower-bd},
\begin{equation}\label{f2-lower-bd}
f_1(\rho)>\3 a_4\lambda^{-\frac{(2m-1)\rho_1}{(1-m)\beta}} \quad \forall\rho\in(\rho_3,\rho_4).
\end{equation}
Hence by \eqref{eq-lower-bound-of-overline-q-in-the-nbd-of-zero}, \eqref{eq-eq-for-v-1-with-f-2-and-some-terms} and \eqref{f2-lower-bd},
\begin{equation}\label{v1'-ineqn}
v_1'(\rho)\le\frac{a_2\overline{w}(\rho)^{1-m}}{\rho}\left[\frac{1-m}{a_2\overline{w}(\rho)^{2-m}}\tilde{v}^2(\rho)
-\frac{\3 a_4\lambda^{-\frac{(2m-1)\rho_1}{(1-m)\beta}}}{\rho}\right]\le -\frac{\delta_0}{\rho^2}<0\quad\forall\rho\in(\rho_3,\rho_4)
\end{equation}
for some constant $\delta_0>0$.
Integrating \eqref{v1'-ineqn} over $(\rho,\rho_4)$, 
\begin{align}\label{eq-align-inequality-for-overline-q-rho-with-v-0a=and-v-1}
v_1(\rho_4)&-v_1(\rho)\le\delta_0\left(\frac{1}{\rho_4}-\frac{1}{\rho}\right)\quad\forall\rho\in (\rho_3,\rho_4)\nonumber\\
\Rightarrow& \qquad \overline{w}_{\rho}(\rho)=\tilde{v}(\rho)=v_0+\rho v_1(\rho)\geq v_0+\rho v_1(\rho_4)+\delta_0-\frac{\delta_0\rho}{\rho_4} \quad\forall\rho\in (\rho_3,\rho_4).
\end{align}
If $\rho_3=0$, then by \eqref{eq-align-inequality-for-overline-q-rho-with-v-0a=and-v-1} and Lemma \ref{lem-limit-of-w-rho-1-over-overline-q--rgo},
\begin{equation*}
v_0=\lim_{\rho\to 0^+}\overline{w}_{\rho}(\rho)\geq v_0+\delta_0>v_0
\end{equation*}
and contradiction arises. Hence $\rho_3>0$. Thus
\begin{equation*}
v_1(\rho_3)=(1+3\3)a_4\lambda^{-\frac{(2m-1)\rho_1}{(1-m)\beta}}.
\end{equation*}
Then by \eqref{v1'-ineqn}, 
\begin{equation*}
v_1(\rho)<v_1(\rho_3)=(1+3\3)a_4\lambda^{-\frac{(2m-1)\rho_1}{(1-m)\beta}} \quad \forall\rho\in(\rho_3,\rho_4)
\end{equation*}
which contradicts \eqref{v1-lower-bd}. Hence no such $\rho_1'>0$ exists and the second inequality in \eqref{v1-upper-lower-bd} follows. By a similar argument the first inequality in \eqref{v1-upper-lower-bd} also holds. Hence \eqref{v1-upper-lower-bd} holds. Since $\3\in (0,1/5)$ is arbitrary, by \eqref{v1-upper-lower-bd},
\begin{equation*}
\overline{w}_{\rho\rho}(0)=\lim_{\rho\to0^+}\frac{v(\rho)-v_0}{\rho}=\lim_{\rho\to 0^+}v_1(\rho)
=\frac{a_3(ma_3-a_1)}{a_2^2}\cdot\lambda^{-\frac{(2m-1)\rho_1}{(1-m)\beta}}
\end{equation*}
and the lemma follows.
\end{proof}

By Lemma \ref{lem-limit-of-w-rho-1-over-overline-q--rgo}, Lemma \ref{lem-limit-of-overline-q-rho-rho-1-over-overline-q--rgo}, \eqref{v-m-q-bar-derivative-relation} and Taylor's expansions for $\overline{w}$ and $\overline{w}_{\rho}$, Theorem \ref{thm-singular-rate-near-zero-of-solutions-1} follows.

\begin{cor}\label{v'<0-cor}
Let $n\ge 3$, $0\le m<\frac{n-2}{n}$, $\rho_1>0$, $\lambda>0$, $\beta>\beta_0^{(m)}$ and  $\phi_m$, $\alpha_m$, $\beta_0^{(m)}$,
be given by \eqref{phi-m-defn}, \eqref{alpha-m-defn}  and \eqref{beta-0-defn} respectively and $v=v^{(m)}$ is a radially symmetric solution of\eqref{elliptic-eqn} in $\R^n\setminus\{0\}$ which satisfies \eqref{blow-up-rate-at-x=0}.  Then 
\begin{equation}\label{v'<0-eqn1}
(v^{(m)})'(r)<0\quad\forall r>0.
\end{equation}
\end{cor}
\begin{proof}
By \eqref{elliptic-eqn} and Lemma \ref{lem-strictly-positivity-of-overline-q-rho-over-zeor-to-infty},
\begin{equation}\label{v-2nd-derivative<0}
(r^{n-1}v(r)^{m-1}v'(r))'=-\alpha_m \left(v(r)+\frac{\beta}{\alpha_m}rv'(r)\right)<0\quad\forall r>0
\end{equation}
By Theorem \ref{thm-singular-rate-near-zero-of-solutions-1} there exists $\xi_0>0$ such that 
\begin{equation}\label{v'<0-near-0}
v'(r)<0\quad\forall 0<r\le\xi_0.
\end{equation}
By \eqref{v-2nd-derivative<0} and \eqref{v'<0-near-0},
\begin{align}\label{v'<0-x-away-from-0}
&r^{n-1}v(r)^{m-1}v'(r)<\xi_0^{n-1}v(\xi_0)^{m-1}v'(\xi_0)<0\quad\forall r>\xi_0\notag\\
\Rightarrow\quad&v'(r)<0\quad\forall r>\xi_0.
\end{align}
By \eqref{v'<0-near-0} and \eqref{v'<0-x-away-from-0}, we get \eqref{v'<0-eqn1} the lemma follows.
\end{proof}

We are now ready for the proof of Theorem \ref{uniqueness-thm}.

\begin{proof}[\textbf{Proof of Theorem \ref{uniqueness-thm}}]
Note that the case $0<m<\frac{n-2}{n}$ and $\beta\ge\frac{\rho_1}{n-2-nm}$ is already proved in \cite{Hu4}. We will give a new proof  which includes all cases of the theorem. By \eqref{elliptic-eqn}, \eqref{blow-up-rate-at-x=0} and integration by parts,
\begin{align}\label{eq-aligned-original-for-g-i-eq}
&r^{n-1}(v_1(r)^{m-1}v_1'(r)-v_2(r)^{m-1}v_2'(r))+\beta r^n(v_1(r)-v_2(r))\notag\\
&\qquad \qquad =\sum_{i=1}^{2}(-1)^{i-1}\xi^{n-2-\frac{m\alpha_m}{\beta}}\left(\xi^{\frac{\alpha_m}{\beta}}v_i(\xi)\right)^{m-1}\xi^{\frac{\alpha_m}{\beta}+1}v_i'(\xi)+\beta\xi^n(v_1(\xi)-v_2(\xi))\notag\\
&\qquad \qquad \qquad \qquad +(n\beta-\alpha_m)\int_{\xi}^{r}(v_1(\rho)-v_2(\rho))\rho^{n-1}\,d\rho,  \qquad \forall r>\xi>0.
\end{align}
By Theorem \ref{thm-singular-rate-near-zero-of-solutions-1}, there exist constants $\xi_0>0$ and $C_0>0$ such that
\begin{equation}\label{eq-control-of-the-first-term-on-the-right-hand-side-of-eq-for-g-1-and-g-2}
\left|\left(\xi^{\frac{\alpha_m}{\beta}}v_i(\xi)\right)^{m-1}\xi^{\frac{\alpha_m}{\beta}+1}v_i'(\xi)\right|\le C_0\quad \forall 0<\xi<\xi_0,i=1,2.
\end{equation}
Since $\beta>\beta_0^{(m)}$, $n-2-\frac{m\alpha_m}{\beta}>0$. Hence by \eqref{eq-control-of-the-first-term-on-the-right-hand-side-of-eq-for-g-1-and-g-2},
\begin{equation}\label{eq-control-of-the-first-term-on-the-right-hand-side-of-eq-for-g-1-and-g-2-2}
\lim_{\xi\to 0}\sum_{i=1}^{2}\left|\xi^{n-2-\frac{m\alpha_m}{\beta}}\left(\xi^{\frac{\alpha_m}{\beta}}v_i(\xi)\right)^{m-1}\xi^{\frac{\alpha_m}{\beta}+1}v_i'(\xi)\right|=0.
\end{equation}
By \eqref{vm-vm'-asymptotic-behaviour} of Theorem \ref{thm-singular-rate-near-zero-of-solutions-1} there exist constants $C>0$ and $r_0>0$ such that 
\begin{equation}\label{v1-v2-weighted-ineqn}
r^{\frac{\alpha_m}{\beta}}|v_1(r)-v_2(r)|\le Cr^{\frac{2\rho_1}{\beta}}\quad\forall 0<r<r_0.
\end{equation}
Hence
\begin{equation}\label{eq-control-of-the-second-term-on-the-right-hand-side-of-eq-for-g-1-and-g-2}
\left|\xi^{n}\left(v_1(\xi)-v_2(\xi)\right)\right|\le C\xi^{n-\frac{\alpha_m}{\beta}+\frac{2\rho_1}{\beta}}=C\xi^{\frac{(n-2-nm)}{(1-m)\beta}(\beta-\beta_0^{(m)})}\cdot\xi^{\frac{\rho_1}{\beta}}\to 0 \qquad \mbox{as $\xi\to 0$}.
\end{equation}
Letting $\xi\to 0$ in \eqref{eq-aligned-original-for-g-i-eq}, by \eqref{eq-control-of-the-first-term-on-the-right-hand-side-of-eq-for-g-1-and-g-2-2} and \eqref{eq-control-of-the-second-term-on-the-right-hand-side-of-eq-for-g-1-and-g-2},
\begin{equation}\label{eq-bound-of0-differences-of-two-solutions-of-under-some-range-of-beta-1}
r^{n-1}(v_1(r)^{m-1}v_1'(r)-v_2(r)^{m-1}v_2'(r))+\beta r^n(v_1(r)-v_2(r))=(n\beta-\alpha_m)\int_0^{r}(v_1(\rho)-v_2(\rho))\rho^{n-1}\,d\rho\quad \forall r>0.
\end{equation}
By Corollary \ref{v'<0-cor},
\begin{equation}\label{vi'<0-locally}
v_i'(r)<0\quad\forall r>0, i=1,2.
\end{equation}
Since $n+\frac{2\rho_1-\alpha_m}{\beta}=\frac{(n-2-nm)}{(1-m)\beta}(\beta-\beta_0^{(m)})+\frac{\rho_1}{\beta}>0$, by \eqref{v1-v2-weighted-ineqn},
\begin{equation}\label{eq-control-of-last-term-by-some-polynomial-345}
\left|(n\beta-\alpha_m)\int_0^{r}(v_1-v_2)(\rho)\rho^{n-1}\,d\rho\right|\leq C\int_0^{r}\rho^{n-\frac{\alpha_m}{\beta}+\frac{2\rho_1}{\beta}-1}\,d\rho=Cr^{n+\frac{2\rho_1-\alpha_m}{\beta}}\quad \forall 0<r<r_0
\end{equation}
for some constant $C>0$. Hence by \eqref{eq-bound-of0-differences-of-two-solutions-of-under-some-range-of-beta-1} and \eqref{eq-control-of-last-term-by-some-polynomial-345},
\begin{equation}\label{eq-aligned-equation-of-difference-g-1-and-g-2-for-general-case-1}
\begin{aligned}
r^{n-1}\left(v_1(r)^{m-1}v_1'(r)-v_2(r)^{m-1}v_2'(r)\right)+\beta r^n(v_1(r)-v_2(r))\leq Cr^{n+\frac{2\rho_1-\alpha_m}{\beta}}\quad \forall 0<r<r_0.
\end{aligned}
\end{equation}
Let 
\begin{equation*}
\mathcal{D}=\left\{0<r<r_0: v_1(r)\geq v_2(r)\right\}.
\end{equation*}
By \eqref{vi'<0-locally} and \eqref{eq-aligned-equation-of-difference-g-1-and-g-2-for-general-case-1} for any $r\in\mathcal{D}$,
\begin{align*}
&v_1(r)^{m-1}v_1'(r)+\beta rv_1(r)\\
&\qquad \qquad \leq v_2(r)^{m-1}v_2'(r)+\beta rv_2(r)+Cr^{1+\frac{2\rho_1-\alpha_m}{\beta}}\leq v_1(r)^{m-1}v_2'(r)+\beta rv_2(r)+Cr^{1+\frac{2\rho_1-\alpha_m}{\beta}}\\
&\Rightarrow \qquad\left(v_1-v_2\right)'(r)+\beta rv_1(r)^{1-m}\left(v_1-v_2\right)(r)\leq Cr^{1+\frac{2\rho_1-\alpha_m}{\beta}}v_1(r)^{1-m}.
\end{align*}
Hence
\begin{align}
&\left(\left(v_1-v_2\right)_+(r)e^{\beta\int_{r_1}^{r}\rho v_1(\rho)^{1-m}\,d\rho}\right)'\leq Cr^{1+\frac{2\rho_1-\alpha_m}{\beta}}v_1(r)^{1-m}e^{\beta\int_{r_1}^{r}\rho v_1(\rho)^{1-m}\,d\rho} \quad \forall 0<r_1<r<r_0\notag\\
&\qquad \Rightarrow\qquad \left(v_1-v_2\right)_+(r_2)\leq \left(v_1-v_2\right)_+(r_1)e^{-\beta\int_{r_1}^{r_2}\rho v_1(\rho)^{1-m}\,d\rho}\notag\\
&\qquad\qquad\qquad\qquad\qquad\qquad +C\frac{\int_{r_1}^{r_2}\rho^{1+\frac{2\rho_1-\alpha_m}{\beta}}v_1(\rho)^{1-m}\left(e^{\beta\int_{r_1}^{\rho}s v_1(s)^{1-m}\,ds}\right)\,d\rho}{e^{\beta\int_{r_1}^{r_2}\rho v_1(\rho)^{1-m}\,d\rho}}\qquad \forall 0<r_1<r_2<r_0.\label{eq-aligned-inequality-of-difference-of-g-1-abd-g-2-on-some-set-1}
\end{align}
Since $r v_1(r)^{1-m}\approx r^{1-(1-m)\frac{\alpha_m}{\beta}}=r^{-1-\frac{\rho_1}{\beta}}$ near $r=0$, both the numerator and denominator of the last term of \eqref{eq-aligned-inequality-of-difference-of-g-1-abd-g-2-on-some-set-1} goes to infinity as $r_1\to 0$. Hence by the l'Hospital rule,
\begin{equation}\label{eq-step-for-removing-the-second-term-in-inequality-for-g-1-minus-g-2}
\lim_{r_1\to 0}\frac{\int_{r_1}^{r_2}\rho^{1+\frac{2\rho_1-\alpha_m}{\beta}}v_1(\rho)^{1-m}\left(e^{\beta\int_{r_1}^{\rho}s v_1(s)^{1-m}\,ds}\right)\,d\rho}{e^{\beta\int_{r_1}^{r_2}\rho v_1(\rho)^{1-m}\,d\rho}}=\lim_{r_1\to 0}\frac{r_1^{\frac{2\rho_1-\alpha_m}{\beta}}}{\beta e^{\beta\int_{r_1}^{r_2}\rho v_1(\rho)^{1-m}\,d\rho}}.
\end{equation}
Since
\begin{equation*}
\lim_{r_1\to 0}\frac{\int_{r_1}^{r_2}\rho\cdot v_1(\rho)^{1-m}\,d\rho}{r_1^{-\frac{\rho_1}{\beta}}}=\lim_{r_1\to 0}\frac{r_1v_1(r_1)^{1-m}}{\frac{\rho_1}{\beta}r_1^{-\frac{\rho_1}{\beta}-1}}=\frac{\beta}{\rho_1}\lambda^{-\frac{\rho_1}{\beta}}\quad\forall 0<r_2<r_0,
\end{equation*}
for any $0<r_2<r_0$ there exists a constant $r_3\in(0,r_2)$ such that
\begin{align}
&\int_{r_1}^{r_2}\rho v_1(\rho)^{1-m}\,d\rho\ge \frac{\beta}{2\rho_1}\lambda^{-\frac{\rho_1}{\beta}}r_1^{-\frac{\rho_1}{\beta}}\qquad \qquad 0<r_1<r_3\notag\\
&\Rightarrow\qquad\left|\frac{r_1^{\theta}}{e^{\beta\int_{r_1}^{r_2}\rho v_1^{1-m}(\rho)\,d\rho}}\right|\leq \frac{r_1^{\theta}}{e^{\frac{\beta^2}{2\rho_1}\left(\lambda r_1\right)^{-\frac{\rho_1}{\beta}}}}\to 0 \quad \mbox{ as }r_1\to 0^+\quad\forall 0<r_2<r_0,\theta\in\R.\label{eq-control-of-the-second-term-in-differneces-equation-as-r-to-zero}
\end{align} 
By \eqref{blow-up-rate-at-x=0} and Corollary \ref{v'<0-cor}
\begin{equation}\label{vi-upper-bd2}
v_i(r)<\lambda^{-\frac{\rho_1}{(1-m)\beta}}r^{-\frac{\alpha_m}{\beta}} \qquad \forall r>0, i=1,2.
\end{equation}
Hence by \eqref{eq-step-for-removing-the-second-term-in-inequality-for-g-1-minus-g-2}, \eqref{eq-control-of-the-second-term-in-differneces-equation-as-r-to-zero}
and \eqref{vi-upper-bd2},
\begin{equation}\label{eq-step-for-removing-the-second-term-in-inequality-for-g-1-minus-g-2-1}
\lim_{r_1\to 0}\frac{\left(v_1-v_2\right)_+(r_1)}{e^{\beta\int_{r_1}^{r_2}\rho v_1^{1-m}\,d\rho}}=0\qquad \mbox{and} \qquad \lim_{r_1\to 0}\frac{\int_{r_1}^{r_2}\rho^{1+\frac{2\rho_1-\alpha_m}{\beta}}v_1(\rho)^{1-m}\left(e^{\beta\int_{r_1}^{\rho}s v_1(s)^{1-m}\,ds}\right)\,d\rho}{e^{\beta\int_{r_1}^{r_2}\rho v_1(\rho)^{1-m}\,d\rho}}=0.
\end{equation}
By \eqref{eq-aligned-inequality-of-difference-of-g-1-abd-g-2-on-some-set-1} and \eqref{eq-step-for-removing-the-second-term-in-inequality-for-g-1-minus-g-2-1}, 
\begin{equation}\label{eq-positive-blanket-of-difference-of-g_1-and-g_2}
\left(v_1-v_2\right)_+(r)\leq 0 \qquad \forall 0\leq r<r_0.
\end{equation}
Similarly
\begin{equation}\label{eq-negative-blanket-of-difference-of-g_1-and-g_2}
\left(v_1-v_2\right)_-(r)\leq 0 \qquad \forall 0\leq r<r_0.
\end{equation}
By \eqref{eq-positive-blanket-of-difference-of-g_1-and-g_2} and \eqref{eq-negative-blanket-of-difference-of-g_1-and-g_2},
\begin{equation}\label{local-uniqueness-eqn}
v_1(r)=v_2(r) \qquad \forall  0\leq r<r_0.
\end{equation}
Then by \eqref{local-uniqueness-eqn} and standard O.D.E. theory,  \eqref{eq-uniqueness-of-solution-g-1-and-g-2-w-r-t-initial-constant-lambda} holds.
\end{proof}

\section{Decay estimates of solutions of the elliptic logarithmic  equation}
\setcounter{equation}{0}
\setcounter{thm}{0}

In this section we will prove the decay rate of solutions of the elliptic logarithmic  equation \eqref{elliptic-log-eqn}.

\begin{lemma}\label{cf-Lemma-2-1-of-Hs2}
Let $n\geq 3$, $\beta\in\R$, $\rho_1>0$ and $\alpha=2\beta+\rho_1$. Let $v=v^{(0)}$ be a radially symmetric solution of \eqref{elliptic-log-eqn} in $\R^n\bs B_1$ and $w(r)=r^2v(r)$. Suppose that there exists a constant $C_0>0$ such that
\begin{equation}\label{eq-upper-bound-of-V-equal-v-and-r-2}
w(r)\leq C_0 \qquad \forall r\geq 1.
\end{equation}
Then, any sequence $\left\{w(r_i)\right\}_{i=1}^{\infty}$, $r_i\to\infty$ as $i\to\infty$, has a subsequence $\left\{w(r_i')\right\}_{i=1}^{\infty}$ such that 
\begin{equation}\label{eq-limit-of-U-r-2-u-as-t-to-infty}
\lim_{i\to\infty}w(r_i')=\left\{
\begin{aligned}
&0\quad\mbox{ or }\quad w_{\infty}\quad\mbox{ if }v\notin L^1(\R^n\setminus B_1)\\
&0\quad\mbox{ or }\quad w_1\quad\mbox{ if }v\in L^1(\R^n\setminus B_1)\quad  \mbox{ and}\quad \beta>0\\
&0\qquad\qquad\quad\mbox{ if }v\in L^1(\R^n\setminus B_1)\quad\mbox{ and }\quad \beta\leq 0
\end{aligned}\right.
\end{equation}
where
\begin{equation*}
w_{\infty}=\frac{2\left(n-2\right)}{\alpha-2\beta} \qquad \mbox{and} \qquad w_1=\frac{2}{\beta}.
\end{equation*}
\end{lemma}
\begin{proof}
We will use a modification of the proof of Lemma 2.1 of \cite{Hs4} to prove the lemma. Let $\left\{r_i\right\}_{i=1}^{\infty}$ be a sequence such that $r_i\to\infty$ as $i\to\infty$. By \eqref{eq-upper-bound-of-V-equal-v-and-r-2} the sequence $\left\{w(r_i)\right\}_{i=1}^{\infty}$ has a subsequence which we may assume without loss of generality to be the sequence itself that converges to some constant $a_0\in[0,C_0]$ as $i\to\infty$. Multiplying \eqref{elliptic-log-eqn} by $r^{n-1}$ and integrating over $(1,r)$,
\begin{equation}\label{eq-for-u-after-integration-once}
v'(r)=a_5\frac{v(r)}{r^{n-1}}-\beta rv^2(r)+\frac{(n\beta-\alpha)}{r^{n-1}}v(r)\int_1^{r}\rho^{n-1}v(\rho)\,d\rho \qquad \forall r\geq 1.
\end{equation}
where 
\begin{equation}\label{a5-defn}
a_5=v(1)^{-1}v'(1)+\beta v(1).
\end{equation} 
Integrating \eqref{eq-for-u-after-integration-once} over $(r,\infty)$, by \eqref{eq-upper-bound-of-V-equal-v-and-r-2} we have
\begin{equation}\label{eq-for-u-after-integration-twice}
v(r)=-a_5\int_{r}^{\infty}s^{1-n}v(s)\,ds+\beta\int_{r}^{\infty}sv^2(s)\,ds+(\alpha-n\beta)\int_{r}^{\infty}s^{1-n}v(s)\left(\int_{1}^{s}\rho^{n-1}v(\rho)\,d\rho\right)\,ds  \qquad \forall r>1.
\end{equation}
By \eqref{eq-upper-bound-of-V-equal-v-and-r-2}, \eqref{eq-for-u-after-integration-twice} and l'Hospital rule,
\begin{align}\label{eq-aligned-equation-for-limit-of-r-2-u-1}
a_0&=\lim_{i\to\infty}r_i^2v(r_i)\notag\\
&=-a_5\lim_{i\to\infty}\frac{\int_{r_i}^{\infty}s^{1-n}v(s)\,ds}{r_i^{-2}}+\beta\lim_{i\to\infty}\frac{\int_{r_i}^{\infty}sv^2(s)\,ds}{r_i^{-2}}+(\alpha -n\beta)\lim_{i\to\infty}\frac{\int_{r_i}^{\infty}s^{1-n}v(s)\left(\int_1^{s}\rho^{n-1}v(\rho)\,d\rho\right)\,ds}{r_i^{-2}}\notag\\
&=\frac{1}{2}\left(-a_5\lim_{i\to\infty}\frac{v(r_i)}{r_i^{n-4}}+\beta\lim_{i\to\infty}\frac{r_iv^2(r_i)}{r_i^{-3}}+(\alpha -n\beta)\lim_{i\to\infty}\frac{r_i^{1-n}v(r_i)\int_1^{r_i}\rho^{n-1}v(\rho)\,d\rho}{r_i^{-3}}\right)\notag\\
&=\frac{1}{2}\left(\beta a_0^2+(\alpha -n\beta)\lim_{i\to\infty}\frac{r_i^2v(r_i)\int_1^{r_i}\rho^{n-1}v(\rho)\,d\rho}{r_i^{n-2}}\right)\notag\\
&=\frac{1}{2}\left(\beta a_0^2+(\alpha -n\beta)a_0\lim_{i\to\infty}\frac{\int_1^{r_i}\rho^{n-1}v(\rho)\,d\rho}{r_i^{n-2}}\right).
\end{align}
If $v\not\in L^1(\R^n\setminus B_1)$, then by \eqref{eq-aligned-equation-for-limit-of-r-2-u-1} and the l'Hospital rule,
\begin{align}\label{eq-limit-a-to-zero-or-w-infty}
&a_0=\frac{1}{2}\left(\beta a_0^2+\frac{(\alpha -n\beta)}{n-2}a_0\lim_{i\to\infty}r_i^2v(r_i)\right)=\frac{\alpha-2\beta}{2(n-2)}a_0^2\notag\\
\Rightarrow\quad&a_0=0 \qquad \mbox{or} \qquad a_0=\frac{2(n-2)}{\alpha-2\beta}=w_{\infty}.
\end{align}
If  $v\in L^1(\R^n\setminus B_1)$, then by \eqref{eq-aligned-equation-for-limit-of-r-2-u-1},
\begin{equation}\label{eq-lim-a-to-zero-or-w-1}
a_0=\frac{\beta}{2}a_0^2 \qquad \Rightarrow \qquad\left\{\begin{aligned}
&a_0=0 \quad \mbox{or} \quad a_0=\frac{2}{\beta}=w_1\quad\mbox{ if }\beta>0\\
&a_0=0\qquad\qquad\qquad\qquad\quad\mbox{ if }\beta\le 0.\end{aligned}\right.
\end{equation} 
By \eqref{eq-limit-a-to-zero-or-w-infty} and \eqref{eq-lim-a-to-zero-or-w-1}, the lemma follows.
\end{proof}

\begin{cor}\label{cf-Corollary-2-3-of-cite-Hs2}
Let $n\geq 3$, $\beta\in\R$, $\rho_1>0$ and $\alpha=2\beta+\rho_1$. Let $v=v^{(0)}$ be a radially symmetric solution of \eqref{elliptic-log-eqn} in $\R^n\bs B_1$ and $w(r)=r^2v(r)$. Suppose that there exist constants $C_0>C_1>0$ such that
\begin{equation*}
C_1\leq w(r)\leq C_0 \qquad \forall r\geq 1.
\end{equation*}
Then \eqref{eq-limit-of-r-2-u-r-at-infty-00} holds.
\end{cor}

\begin{lemma}\label{cf-Lemma-2-4-of-Hs2}
Let $n\geq 3$, $\rho_1>0$, $\beta>\beta_{1}^{(0)}:=\frac{\rho_1}{n-2}$ and $\alpha=2\beta+\rho_1$. Let $v=v^{(0)}$ be a radially symmetric solution of \eqref{elliptic-log-eqn} in $\R^n\setminus B_1$  and $w(r)=r^2v(r)$. Then there exists a constant $C_1>0$ such that
\begin{equation}\label{eq-lower-bound-of-w-in-case-of-alpha-bigger-tnan-n-beta}
w(r)\ge C_1 \qquad \forall r\ge 1. 
\end{equation}
\end{lemma}
\begin{proof}
By \eqref{eq-for-u-after-integration-once},
\begin{equation}\label{v'-lower-ineqn}
v'(r)+\beta rv^2(r)+|a_5|r^{1-n}v(r)\ge 0\quad \forall r\ge 1
\end{equation}
where $a_5$ is given by \eqref{a5-defn}. Let $H(r)=e^{-\frac{|a_5|}{n-2}r^{2-n}}v(r)$. Then by \eqref{v'-lower-ineqn},
\begin{align}
&H'(r)\ge-\beta  e^{\frac{|a_5|}{n-2}r^{2-n}}rH(r)^2\ge -\beta  e^{\frac{|a_5|}{n-2}}rH(r)^2\quad \forall r\ge 1\label{eq-integral-equationfor-v-of-log-diff-after-integration-8}\\
\Rightarrow\quad &-H(r)^{-2}H'(r)\le\beta  e^{\frac{|a_5|}{n-2}}r\qquad \qquad \qquad \qquad \forall r\ge 1\notag\\
\Rightarrow\quad &v(r)\ge H(r)\ge\left(\frac{\beta  e^{\frac{|a_5|}{n-2}}}{2}r^2+H(1)^{-1}\right)^{-1}\qquad \quad\forall r\ge 1\label{v-H-lower-bd}.
\end{align}
By \eqref{v-H-lower-bd} there exists a constant $C_1>0$ such that
\eqref{eq-lower-bound-of-w-in-case-of-alpha-bigger-tnan-n-beta} holds and the lemma follows.
\end{proof}

\begin{lemma}\label{cf-Lemma-2-2-of-cite-Hs3}
Let $n\geq 3$, $\rho_1>0$, $\beta\leq\beta_1^{(0)}:=\frac{\rho_1}{n-2}$ and $\alpha=2\beta+\rho_1$.  Let $v=v^{(0)}$ be a radially symmetric solution of \eqref{elliptic-log-eqn} in $\R^n\setminus B_1$  and $w(r)=r^2v(r)$. Then there exists a constant $C_1>0$ such that \eqref{eq-lower-bound-of-w-in-case-of-alpha-bigger-tnan-n-beta} holds.
\end{lemma}
\begin{proof}
As observed in \cite{Hs1}, $w$ satisfies
\begin{equation}\label{eq-for-w-by-direct-computation}
\left(\frac{w'}{w}\right)'+\frac{n-1}{r}\cdot\frac{w'}{w}+\frac{\beta}{r}w'+\frac{(\alpha-2\beta)w-2(n-2)}{r^2}=0\quad \forall r\ge 1.
\end{equation}
Multiplying \eqref{eq-for-w-by-direct-computation} by $r^{n-1}$ and  integrating over $(1,r)$, $r>1>0$, 
\begin{align}
&\frac{r^{n-1}w'(r)}{w(r)}+\beta r^{n-2}w(r)\notag\\
&\qquad\quad =\frac{w'(1)}{w(1)}+\beta w(1)+(n\beta-\alpha)\int_1^{r}\rho^{n-3}w(\rho)\,d\rho+2(r^{n-2}-1)\quad\forall r>1\notag\\
\Rightarrow\quad &\frac{rw'(r)}{w(r)}=2-\beta w(r)+\frac{(n\beta-\alpha)}{r^{n-2}}\int_{0}^{r}\rho^{n-3}w(\rho)\,d\rho+\frac{a_6}{r^{n-2}} \quad \forall r>1\label{eq-for-w-after-integration-once-23}
\end{align}
where $a_6=w(1)^{-1}w'(1)+\beta w(1)-2$. Let 
\begin{equation*}
\mathcal{B}=\left\{r>1:w(r)\leq\frac{1}{\rho_1}\right\}.
\end{equation*} 
If there is a constant $R_0>1$ such that $\mathcal{B}\cap[R_0,\infty)=\emptyset$, then
\begin{equation*}
w(r)\geq \frac{1}{\rho_1} \qquad \forall r\geq R_0
\end{equation*}
and \eqref{eq-lower-bound-of-w-in-case-of-alpha-bigger-tnan-n-beta} follows. Hence we may assume that 
\begin{equation}\label{eq-existence-of-point-in-B-and-r-to-infty-for-all-r-greater-tahn-1}
\mathcal{B}\cap[R_1,\infty)\neq\emptyset \qquad \forall R_1>1.
\end{equation}
If $\int_1^{\infty}\rho^{n-3}w(\rho)\,d\rho=\infty$ holds, then
by \eqref{eq-for-w-after-integration-once-23} and the l'Hospital rule,
\begin{equation*}
\liminf_{\begin{subarray}{c}r\to\infty\\r\in\mathcal{B}\end{subarray}}\frac{rw'(r)}{w(r)}\geq 2-\frac{\beta}{\rho_1}-\frac{\alpha-n\beta}{n-2}\cdot\frac{1}{\rho_1}=2-\frac{1}{n-2}>0.
\end{equation*}
If $\int_1^{\infty}\rho^{n-3}w(\rho)\,d\rho<\infty$ holds, then by \eqref{eq-for-w-after-integration-once-23},
\begin{equation*}
\liminf_{\begin{subarray}{c}r\to\infty\\r\in\mathcal{B}\end{subarray}}\frac{rw'(r)}{w(r)}\geq 2-\frac{\beta}{\rho_1}>2-\frac{1}{n-2}>0.
\end{equation*}
Hence in both cases there exists a constant $R_2\in\mathcal{B}$ such that 
\begin{equation}\label{eq-positivity-of-w-prime-over-R-3-to-infty-and-in-B}
w'(r)>0 \quad\forall r\in\mathcal{B}\cap [R_2,\infty). 
\end{equation}
Suppose that there exists a constant $R_3>R_2$ such that $R_3\notin\mathcal{B}$. Let
\begin{equation*}
R_4=\sup\left\{r_1>R_3:w(r)>\frac{1}{\rho_1}\quad\forall R_3\le r<r_1\right\}.
\end{equation*}
By \eqref{eq-existence-of-point-in-B-and-r-to-infty-for-all-r-greater-tahn-1},
\begin{equation*}
R_4<\infty \qquad \Rightarrow \qquad w(R_4)=\frac{1}{\rho_1}, \quad R_4\in \mathcal{B}\quad\mbox{ and }\quad w'(R_4)\leq 0
\end{equation*}
which contradicts \eqref{eq-positivity-of-w-prime-over-R-3-to-infty-and-in-B}. Thus no such point $R_3$ exists. Hence
\begin{equation}\label{B-contain-right-interval}
[R_2,\infty)\subset\mathcal{B}.
\end{equation}
By \eqref{eq-positivity-of-w-prime-over-R-3-to-infty-and-in-B} and \eqref{B-contain-right-interval},
\begin{equation*}
w(r)\ge w(R_2)\quad\forall r\ge R_2
\end{equation*}
and the lemma follows.
\end{proof}

\begin{lemma}\label{cf-Lemma-2-5-of-cite-Hs2}
Let $n\geq 3$, $\rho_1>0$, $\beta\leq\beta_1^{(0)}:=\frac{\rho_1}{n-2}$ and $\alpha=2\beta+\rho_1$. Let $v=v^{(0)}$ be a radially symmetric solution of \eqref{elliptic-log-eqn} in $\R^n\setminus B_1$ and $w(r)=r^2v(r)$. Then there exists a constant $C_0>0$ such that
\eqref{eq-upper-bound-of-V-equal-v-and-r-2} holds.
\end{lemma}
\begin{proof}
By Corollary \ref{v'<0-cor} $v$ satisfies \eqref{v'<0-eqn1}. Since $\alpha\ge n\beta$, by \eqref{elliptic-log-eqn}, \eqref{log v-soln-x=0-rate}, \eqref{v'<0-eqn1} and Lemma \ref{cf-Lemma-2-2-of-cite-Hs3},
\begin{align}
r^{n-1}\frac{v'(r)}{v(r)}+\beta r^n v(r)=&\frac{v'(1)}{v(1)}+\beta v(1)-(\alpha-n\beta)\int_1^{r}\rho^{n-1}v(\rho)\,d\rho\quad \forall r>1\notag\\
\le&\frac{v'(1)}{v(1)}+\beta v(1)-(\alpha-n\beta)\int_1^{r}\rho^{n-1}v(r)\,d\rho\quad \forall r>1\notag\\
\le&a_5-\frac{(\alpha-n\beta)}{n}r^nv(r) \quad \qquad\qquad\qquad \qquad \forall r>1\notag\\
\Rightarrow \quad r^{n-1}\frac{v'(r)}{v(r)}+\frac{\alpha}{n} r^n v(r)\le&a_5\quad \forall r>1\notag\\
\Rightarrow \qquad\qquad\,\,\,\frac{v'(r)}{v(r)^2}+\frac{\alpha}{n} r\le&\frac{|a_5|}{r^{n-1}v(r)}\le\frac{C_2}{r^{n-3}}\leq C_2 \qquad\forall r>1 
\label{v-ineqn1}
\end{align}
for some constant $C_2>0$ where $a_5$ is given by \eqref{a5-defn}. Integrating \eqref{v-ineqn1} over $(1,r)$, 
\begin{equation}\label{eq-lower-bound-of-1-over-u-by-r-square-and-something}
\begin{aligned}
\frac{1}{v(r)}\geq \frac{\alpha r^2}{2n}-C_2(r-1)-\frac{\alpha}{2n}+\frac{1}{v(1)}\geq \frac{\alpha r^2}{4n}+C_3 \quad \forall r>\max\left(1,\frac{4nC_2}{\alpha}\right)
\end{aligned}
\end{equation}
where $C_3=C_2-\frac{\alpha}{2n}+\frac{1}{v(1)}$. Then by \eqref{v'<0-eqn1} and \eqref{eq-lower-bound-of-1-over-u-by-r-square-and-something},
\eqref{eq-upper-bound-of-V-equal-v-and-r-2} holds for some constant $C_0>0$ and the lemma follows.
\end{proof}

\begin{lemma}[cf. Lemma 2.6 of \cite{Hs4}]\label{cf-Lemma-2-6-of-cite-Hs2}
Let $n\geq 3$, $\rho_1>0$, $\beta>\beta_1^{(0)}=\frac{\rho_1}{n-2}$ and $\alpha=2\beta+\rho_1$. Let $v=v^{(0)}$ be a radially symmetric solution of \eqref{elliptic-log-eqn} in $\R^n\setminus B_1$ and $w(r)=r^2v(r)$. Then there exists a constant $C_0>0$ such that \eqref{eq-upper-bound-of-V-equal-v-and-r-2} holds. 
\end{lemma}
\begin{proof}
The proof of the lemma is similar to the proof of Lemma 2.6 of \cite{Hs4}. For the sake of completeness we will give a sketch of the proof here.
Let $A=\left\{r\in[1,\infty):w'(r)\geq 0\right\}$. If there exists a constant $R_0>1$ such that $A\cap[R_0,\infty)=\emptyset$. Then $w'(r)<0$ for all $r\geq R_0$ and \eqref{eq-upper-bound-of-V-equal-v-and-r-2} holds with $C_0=\max_{1\geq r\geq R_0}w(r)$.

We next suppose that $A\cap[R_0,\infty)\neq\emptyset$ for any $R_0>1$.  
By Lemma \ref{cf-Lemma-2-4-of-Hs2} and the l'Hospital rule, 
\begin{equation*}
\limsup_{r\in A,\,r\to\infty}\frac{\int_1^rz^{n-1}v(z)\,dz}{r^nv(r)}=\limsup_{r\in A,\,r\to\infty}\frac{\int_1^rz^{n-1}v(z)\,dz}{r^{n-2}w(r)}\leq \limsup_{r\in A,\,r\to\infty}\frac{r^{n-1}v(r)}{(n-2)r^{n-3}w(r)+r^{n-2}w'(r)}\le\frac{1}{n-2}.
\end{equation*}
Hence there exists a constant $R_1>1$ such that
\begin{equation}\label{v-l1-upper-bd}
\int_1^rz^{n-1}v(z)\,dz\le\left(\frac{1}{n-2}+\frac{\rho_1}{2(n-2)(n\beta-\alpha)}\right)r^nv(r)\quad\forall r\ge R_1, r\in A.
\end{equation}
By \eqref{eq-for-u-after-integration-once} and \eqref{v-l1-upper-bd} for any  $r\geq R_1,r\in A$,
\begin{equation}\label{rv'/v-upper-bd}
\begin{aligned}
\frac{rv'(r)}{v(r)}&\le\frac{a_5}{r^{n-2}}-\beta r^2v(r)+(n\beta-\alpha)\left(\frac{1}{n-2}+\frac{\rho_1}{2(n-2)(n\beta-\alpha)}\right)r^2v(r)
\leq \frac{a_5}{R_1^{n-2}}-\frac{\rho_1}{2(n-2)}w(r).
\end{aligned}
\end{equation}
where $a_5$ is given by \eqref{a5-defn}. Hence by \eqref{rv'/v-upper-bd},
\begin{align}
0\leq &w'(r)=\frac{2w(r)}{r}\left(1+\frac{1}{2}\frac{rv'(r)}{v(r)}\right)\leq \frac{2w(r)}{r}\left(1+\frac{a_5}{2R_1^{n-2}}-\frac{\rho_1}{4(n-2)}w(r)\right)\quad \forall r\geq R_1,r\in A\notag\\
\Rightarrow \quad w(r)&\leq C_3 \qquad\qquad\qquad\forall r\geq R_1,r\in A\label{w-upper-bd-2}
\end{align}
for some constant $C_3>0$.
Since $w'(r)<0$ for any $r\in[R_1,\infty)\bs A$, by \eqref{w-upper-bd-2} and the same argument as the proof of Lemma 2.6 of \cite{Hs4}
\eqref{eq-upper-bound-of-V-equal-v-and-r-2} follows.
\end{proof}

\begin{proof}[\textbf{Proof of Theorem \ref{thm-behaviour-of-v-m-at-infty}}]
If $n\beta>\alpha$, by Corollary \ref{cf-Corollary-2-3-of-cite-Hs2}, Lemma \ref{cf-Lemma-2-4-of-Hs2} and Lemma \ref{cf-Lemma-2-6-of-cite-Hs2}, \eqref{eq-limit-of-r-2-u-r-at-infty-00} follows. If $\alpha\geq n\beta$, by Lemma \ref{cf-Lemma-2-2-of-cite-Hs3}, Lemma \ref{cf-Lemma-2-5-of-cite-Hs2} and Corollary \ref{cf-Corollary-2-3-of-cite-Hs2}, \eqref{eq-limit-of-r-2-u-r-at-infty-00} follows.
\end{proof}

\section{Singular limits of solutions}
\setcounter{equation}{0}
\setcounter{thm}{0}

In this section we will prove the singular limits of solutions of \eqref{fde} and \eqref{elliptic-eqn} as $m\to 0^+$. We first start with a lemma.

\begin{lemma}\label{lemma-upper-bound-of-radially-symmetric-solution-1}
Let $n\geq 3$, $0<\2{m}_0<\frac{n-2}{n}$, $\rho_1>0$, $\lambda>0$, $\beta\ge\beta_0^{(\2{m}_0)}$ and $\alpha_m=\frac{2\beta+\rho_1}{1-m}$. For any $0<m<\2{m}_0$, let $v^{(m)}$ be the radially symmetric solution 
of \eqref{elliptic-eqn} in $\R^n\setminus\{0\}$ which satisfies \eqref{blow-up-rate-at-x=0} given by Theorem \ref{elliptic-eqn-existence-thm}.
Then there exists a constant $m_0\in (0,\2{m}_0)$ such that 
\begin{equation}\label{eq-condition-of-g-lambda-radially-symmetric-solution-by-inequality-2}
\lambda^{-\frac{\rho_1}{(1-m)\beta}} \leq r^{\frac{\alpha_m}{\beta}}v^{(m)}(r)\leq \lambda^{-\frac{\rho_1}{(1-m)\beta}} \exp\left(C_m\lambda^{\frac{\rho_1}{\beta}}r^{\frac{\rho_1}{\beta}}\right) \qquad \forall |x|=r>0, 0<m<m_0
\end{equation}
holds where 
\begin{equation}\label{eq-constant-C-0-of-upper-bound-of-g-lambda}
C_m=\frac{\alpha_m}{\rho_1\beta}\left(n-2-\frac{m\alpha_m}{\beta}\right).
\end{equation}
\end{lemma}
\begin{proof}
We will use a modification of the technique of \cite{Hu4} to prove the theorem. Note that
\begin{equation*}
\beta>\frac{m\rho_1}{n-2-nm} \qquad \forall 0<m<\overline{m}_0.
\end{equation*}
By the proof of Theorem 1.1 of \cite{Hu4}, for any $i\in\Z^+$, $0<m<\2{m}_0$, there exists a radially symmetric solution $v_i$ of
\begin{equation*}
\begin{cases}
\begin{aligned}
&\La\phi_m(v)+\alpha_mv+\beta x\cdot\nabla v=0, \quad v>0,\qquad \qquad \mbox{ in }\R^n\bs B_{\frac{1}{i}}\\
&v_i(1/i)=\lambda^{-\frac{\rho_1}{(1-m)\beta}}i^{\frac{\alpha_m}{\beta}}, \\
& v_i'(1/i)=-\frac{\alpha_m}{\beta}\lambda^{-\frac{\rho_1}{(1-m)\beta}}i^{\frac{\alpha_m}{\beta}+1}
\end{aligned}
\end{cases}
\end{equation*}
which satisfies
\begin{equation}\label{eq-aligned-inequality-for-lower-bound-of-g-i}
v_i'(r)<0\quad\forall r\ge \frac{1}{i}\quad\mbox{ and }\quad v_i(r)\ge\lambda^{-\frac{\rho_1}{(1-m)\beta}}r^{-\frac{\alpha_m}{\beta}}\quad\forall r\ge \frac{1}{i}.
\end{equation}
Moreover the sequence $v_i$ has a subsequence which we may assume without loss of generality to be the sequence $v_i$ itself that converges uniformly in $C^2(K)$ for any compact subset $K$ of $\R^n\bs\{0\}$  to $v=v^{(m)}$ as $i\to\infty$. Let  $w_i(r)=r^{\frac{\alpha_m}{\beta}}v_i(r)$, $s=\log r$ and $z_i(s)=w_i^{-1}\frac{\1 w_i}{\1 s}$. Then by the proof of Theorem 1.1 of \cite{Hu4} (cf. \cite{Hs2}) and \eqref{eq-aligned-inequality-for-lower-bound-of-g-i}, 
\begin{align}
&\left(\frac{w_{i,r}}{w_i}\right)_r+\frac{n-1-\frac{2m\alpha_m}{\beta}}{r}\cdot\frac{w_{i,r}}{w_i}+m\left(\frac{w_{i,r}}{w_i}\right)^2+\frac{\beta r^{-1-\frac{\rho_1}{\beta}}w_{i,r}}{w_i^m}=\frac{\alpha_m}{\beta}\cdot\frac{n-2-\frac{m\alpha_m}{\beta}}{r^2} \quad \forall r>1/i,i\in\mathbb{N}\notag\\
\Rightarrow\quad& z_{i,s}+\left(n-2-\frac{2m\alpha_m}{\beta}\right)z_i+mz_i^2+\beta e^{-\frac{\rho_1}{\beta}s}w_i^{1-m}z_i=\rho_1C_m \quad \forall s>-\log i,i\in\mathbb{N}.\label{zi-ineqn1}
\end{align}
We now choose $m_0\in(0,\overline{m}_0)$ such that 
\begin{equation*}
n-2-\frac{2m\alpha_m}{\beta}>0 \qquad \forall 0<m<m_0.
\end{equation*}
Since by the proof of Theorem 1.1. of \cite{Hu4} $z_i(s)=w_i^{-1}\frac{\1 w_i}{\1 s}\ge 0$ for all $s>-\log i$, by \eqref{eq-aligned-inequality-for-lower-bound-of-g-i} and \eqref{zi-ineqn1},
\begin{equation}\label{label-eq-for-z-i}
z_{i,s}+\beta \lambda^{-\frac{\rho_1}{\beta}}e^{-\frac{\rho_1}{\beta}s}z_i\le\rho_1C_m\qquad \forall s>-\log i, \,i\in\Z^+, 0<m<m_0
\end{equation}
By \eqref{label-eq-for-z-i} and an argument similar to the proof of Theorem 1.1 in \cite{Hu4},
\begin{align}
&z_i(s)\leq \frac{\rho_1C_m}{\beta}\lambda^{\frac{\rho_1}{\beta}}e^{\frac{\rho_1}{\beta}s}\qquad\qquad\qquad \forall s>-\log i,\,i\in\Z^+, 0<m<m_0\notag\\
\Rightarrow\quad&w_i(r)\leq \lambda^{-\frac{\rho_1}{(1-m)\beta}}\exp\left\{C_m\lambda^{\frac{\rho_1}{\beta}}r^{\frac{\rho_1}{\beta}}\right\} \qquad \forall r>1/i,\,i\in Z^+, 0<m<m_0\nonumber\\
\Rightarrow\quad& v_i(r)\leq \lambda^{-\frac{\rho_1}{(1-m)\beta}}r^{-\frac{\alpha_m}{\beta}}\exp\left\{C_m\lambda^{\frac{\rho_1}{\beta}}r^{\frac{\rho_1}{\beta}}\right\} \quad\forall r>1/i,\,i\in Z^+, 0<m<m_0.\label{vi-uniform-upper-bd}
\end{align}
Letting $i\to\infty$ in \eqref{eq-aligned-inequality-for-lower-bound-of-g-i} and \eqref{vi-uniform-upper-bd}, we get \eqref{eq-condition-of-g-lambda-radially-symmetric-solution-by-inequality-2} and the lemma follows
\end{proof}
\begin{proof}[\textbf{Proof of Theorem \ref{elliptic-singular-limit-thm}}]
Let $m_0\in (0,\2{m}_0)$ be given by Lemma \ref{lemma-upper-bound-of-radially-symmetric-solution-1}.
Let $\{m_i\}_{i=1}^{\infty}$, $0<m_i<m_0$ for all $i\in\Z^+$, be a sequence such that $m_i\to 0$ as $i\to\infty$. Let $R_2>R_1>0$. By \eqref{eq-condition-of-g-lambda-radially-symmetric-solution-by-inequality-2}, 
\begin{equation}\label{vm-upper-lower-bd1}
M_1(R_2)\le v^{(m)}(x)\le M_2(R_1,R_2) \qquad \forall R_1\le |x|\le R_2
\end{equation}
where 
\begin{equation*}
\left\{\begin{aligned}
&M_1(R_2)=\min\left(\lambda^{-\frac{n\rho_1}{2\beta}}, \lambda^{-\frac{\rho_1}{\beta}}\right)\min\left(R_2^{-\frac{n}{2}\left(2+\frac{\rho_1}{\beta}\right)},R_2^{-\left(2+\frac{\rho_1}{\beta}\right)}\right)\\
&M_2(R_1,R_2)=\max\left(\lambda^{-\frac{n\rho_1}{2\beta}}, \lambda^{-\frac{\rho_1}{\beta}}\right)\max\left(R_1^{-\frac{n}{2}\left(2+\frac{\rho_1}{\beta}\right)},R_1^{-\left(2+\frac{\rho_1}{\beta}\right)}\right)\exp\left(\frac{n(n-2)(2\beta+\rho_1)}{2\rho_1\beta}\lambda^{\frac{\rho_1}{\beta}}R_2^{\frac{\rho_1}{\beta}}\right).
\end{aligned}\right.
\end{equation*}
By \eqref{vm-upper-lower-bd1} and the mean value theorem, for any $0<m<m_0$ there exists $r_m\in (1,2)$ such that
\begin{equation}\label{vm'-uniformly-bd}
|(v^{(m)})'(r_m)|=|v^{(m)}(2)-v^{(m)}(1)|\le 2M_2(1,2).
\end{equation}
Multiplying \eqref{elliptic-eqn} by $r^{n-1}$ and integrating over $(r_m,r)$, $R_1\le r\le R_2$,
\begin{align}\label{vm-integral-eqn2}
&r^{n-1}(v^{(m)}(r))^{m-1}(v^{(m)})'(r)\notag\\
=&r_m^{n-1}(v^{(m)}(r_m))^{m-1}(v^{(m)})'(r_m)+\beta r_m^nv^{(m)}(r_m)-\beta r^nv^{(m)}(r)+(n\beta-\alpha_m)\int_{r_m}^{r}v^{(m)}(\rho)\rho^{n-1}\,d\rho.
\end{align}
By \eqref{vm-upper-lower-bd1},  \eqref{vm'-uniformly-bd} and \eqref{vm-integral-eqn2}, for any $R_2>R_1>0$ there exists a constant $M_3(R_1,R_2)>0$ such that
\begin{align}
&|(v^{(m)})'(r)|\le M_3(R_1,R_2)\qquad \qquad \forall R_1\le r\le R_2, 0<m<m_0\label{vm'-bd2}\\
\Rightarrow\quad&|v^{(m)}(r_1)-v^{(m)}(r_2)|\le M_3(R_1,R_2)|r_1-r_2|\quad\forall r_1,r_2\in [R_1,R_2], 0<m<m_0.\label{vm'-uniformly-bd2}
\end{align}
By \eqref{elliptic-eqn}, \eqref{vm-upper-lower-bd1} and \eqref{vm'-bd2}, for any $R_1\le r\le R_2$ and $0<m_0<\2{m}_0$,
\begin{align}
&(v^{(m)})''(r)=(1-m)(v^{(m)}(r))^{-1}(v^{(m)})'(r)^2-\alpha v^{(m)}(r)^{2-m}\notag\\
&\qquad \qquad \qquad \qquad -\beta rv^{(m)}(r)^{1-m}(v^{(m)})'(r)-(n-1)r^{-1}(v^{(m)})'(r)
\label{vm''-eqn}\\
\Rightarrow\quad&|(v^{(m)})''(r)|\le M_4(R_1,R_2)\qquad \qquad \forall R_1\le r\le R_2, 0<m<m_0\notag\\
\Rightarrow\quad&|(v^{(m)})'(r_1)-(v^{(m)})'(r_2)|\le M_4(R_1,R_2)|r_1-r_2|\quad\forall r_1,r_2\in [R_1,R_2], 0<m<m_0.\label{vm''-bd1}
\end{align}
for some constant $M_4(R_1,R_2)>0$. By differentiating \eqref{elliptic-eqn} with respect to $r>0$ and repeating the above argument, there exists a constant $M_5(R_1,R_2)>0$ such that
\begin{equation}\label{vm'''-bd1}
|(v^{(m)})'''(r)|\le M_5(R_1,R_2)\quad\mbox{ and }\quad |(v^{(m)})''(r_1)-(v^{(m)})''(r_2)|\le M_5(R_1,R_2)|r_1-r_2|\quad\forall r,r_1,r_2\in [R_1,R_2]
\end{equation} 
holds for any $0<m<m_0$.
By \eqref{vm-upper-lower-bd1}, \eqref{vm'-bd2}, \eqref{vm'-uniformly-bd2}, \eqref{vm''-bd1} and \eqref{vm'''-bd1}, the sequence $\{v^{(m_i)}\}_{i=1}^{\infty}$ is equi-Holder continuous in $C^2(K)$ for any compact subset $K$ of $\R^n\setminus\{0\}$. By the Ascoli Theorem and a diagonalization argument the sequence $\{v^{(m_i)}\}_{i=1}^{\infty}$ has a subsequence which we may assume without loss of generality to be the sequence itself that converges uniformly in $C^2(K)$ for any compact subset $K$ of $\R^n\setminus\{0\}$ to some positive function $v\in C^2(\R^n\setminus\{0\})$ as $i\to\infty$.

Putting $m=m_i$ in \eqref{vm''-eqn} and letting $i\to\infty$, 
\begin{equation*}
v''(r)=(v(r))^{-1}v'(r)^2-\alpha v(r)^{2-1}-\beta v(r)v'(r),\quad v>0,\quad\mbox{ in }\R^n\setminus\{0\}
\end{equation*}
and hence $v$ satisfies \eqref{elliptic-log-eqn}.
Letting $m=m_i\to 0$ in \eqref{eq-condition-of-g-lambda-radially-symmetric-solution-by-inequality-2}, 
\begin{align*}
&\lambda^{-\frac{\rho_1}{\beta}}\leq |x|^{\frac{\alpha}{\beta}}v(x)\leq \lambda^{-\frac{\rho_1}{\beta}} \exp\left(C_0\lambda^{\frac{\rho_1}{\beta}}|x|^{\frac{\rho_1}{\beta}}\right) \qquad \forall x\in\R^n\bs\{0\}\\
\Rightarrow \quad& \lim_{|x|\to 0}|x|^{\frac{\alpha}{\beta}}v(x)=\lambda^{-\frac{\rho_1}{\beta}}
\end{align*}
where $C_0=\frac{(2\beta+\rho_1)(n-2)}{\rho_1\beta}$.
Then by Theorem \ref{uniqueness-thm} $v$ is the unique solution of \eqref{elliptic-log-eqn} which satisfies \eqref{log v-soln-x=0-rate}. Since the sequence $\{m_i\}_{i=1}^{\infty}$ is arbitrary, $v^{(m)}$ converges uniformly in $C^2(K)$ for any compact subset of $\R^n\bs\{0\}$ to the unique solution $v$ of \eqref{elliptic-log-eqn} which satisfies  \eqref{log v-soln-x=0-rate} as $m\to 0^+$ and the theorem follows.
\end{proof}

\begin{proof}[\textbf{Proof of Theorem \ref{singular-log-diffusion-eqn-uniqueness-thm}:}]
We will use a modification of the proof of Lemma 2.5 of \cite{Hu2} to prove the theorem.
Let $h\in C^{\infty}_0(\R^n)$, $0\leq h\leq 1$, $h(x)=1$ for $|x|\leq 1$ and $h(x)=0$ for $|x|\geq 2$. Let $\eta(x)=h(x)^{4}$ and $\eta_R(x)=\eta(x/R)$ for any  $R>0$. For any $R>3\epsilon>0$, let 
\begin{equation*}
\eta_{\epsilon,R}(x)=(1-\eta(x/\epsilon))\eta_{R}(x).
\end{equation*}
Then
\begin{equation*}
\left\{\begin{aligned}
&\eta_{\epsilon,R}=0 \qquad\,\forall |x|\le\epsilon\mbox{ or } |x|\ge 2R\\
&\eta_{\epsilon,R}(x)=1\quad\forall 2\epsilon\le |x|\le R
\end{aligned}\right.
\end{equation*}
and
\begin{equation}\label{eq-behaviour-of-La-eta-epsilon-R-on-epsion-and-R-region}
\left|\La\eta_{\epsilon,R}(x)\right|\leq \frac{C_1}{\epsilon^2} \qquad \forall \epsilon\le |x|\le 2\epsilon, \qquad \qquad \left|\La\eta_{\epsilon,R}(x)\right|\leq \frac{C_1}{R^2} \qquad \forall R\le |x|\le 2R
\end{equation}
for some constant $C_1>0$. By Kato's inequality \cite{Ka},
\begin{align}
\frac{\partial}{\partial t}\int_{\R^n}\left(u_1-u_2\right)_+(x,t)\eta_{\epsilon,R}(x)\,dx\le&\int_{\R^n}\left(\log u_1-\log u_2\right)_+(x,t)\La\eta_{\epsilon,R}(x)\,dx\notag\\
\le&\frac{C_1}{\epsilon^2}\int_{\epsilon\le |x|\leq 2\epsilon}\left(\log u_1-\log u_2\right)_+(x,t)\,dx\notag\\
&\qquad +\int_{R\leq |x|\leq 2R}\left(\log u_1-\log u_2\right)_+(x,t)\left|\La \eta_{R}(x)\right|\,dx
\label{eq-aligned-after-kato-inequality-wiht-eta-epsilon-R-2}
\end{align}
By \eqref{log v-soln-x=0-rate} and Lemma \ref{lem-strictly-positivity-of-overline-q-rho-over-zeor-to-infty} there exists a constant $\epsilon_1>0$ such that
\begin{equation}\label{eq-bound-of-u-lambda-sb-i-by-1-over-2-lambda-to0-something-andp2-2-lambda-to-something}
\lambda_i^{-\frac{\rho_1}{\beta}}\le|x|^{\frac{\alpha}{\beta}}v_{\lambda_i}(x)\le 2\lambda_i^{-\frac{\rho_1}{\beta}},\qquad \forall |x|\leq \epsilon_1,i=1,2.
\end{equation}
Then by \eqref{eq-trapping-solution-u-i-between-V-lambda-1-and-V-lambda-2} and \eqref{eq-bound-of-u-lambda-sb-i-by-1-over-2-lambda-to0-something-andp2-2-lambda-to-something},
\begin{align}
(\log u_1-\log u_2)_+(x,t)\le&\log\left(2\lambda_2^{-\frac{\rho_1}{\beta}}\left((T-t)|x|\right)^{-\alpha/\beta}\right)-\log\left(\lambda_1^{-\frac{\rho_1}{\beta}}\left((T-t)|x|\right)^{-\alpha/\beta}\right)\notag\\
\le&\frac{\rho_1}{\beta}\log\left(\frac{\lambda_1}{\lambda_2}\right)+\log 2 \qquad  \forall |x|\leq \epsilon_1/T^{\beta}, 0<t<T\notag\\
\Rightarrow\quad\left|\frac{1}{\epsilon^2}\int_{\epsilon\le |x|\le 2\epsilon}\left(\log u_1-\log u_2\right)_+(x,t)\,dx\right|
\le&2^n\left(\frac{\rho_1}{\beta}\log\left(\frac{\lambda_1}{\lambda_2}\right)+\log 2\right)\omega_n\epsilon^{n-2}\quad \forall 0<\epsilon\le\frac{\epsilon_1}{2T^{\beta}},0<t<T\notag\\
\to&0 \qquad \qquad\forall 0<t<T\qquad \qquad \mbox{ as }\epsilon\to 0\label{integral-near-origin-goes-to-0}
\end{align}
where $\omega_n$ is the surface area of the unit sphere $S^{n-1}$ in $\R^n$. Letting $\epsilon\to 0$ in \eqref{eq-aligned-after-kato-inequality-wiht-eta-epsilon-R-2}, by \eqref{integral-near-origin-goes-to-0} we get
\begin{equation}\label{eq-process-for-L-1-contraction-after-letting-epsilon-to-zero-1}
\frac{\partial}{\partial t}\int_{\R^n}\left(u_1-u_2\right)_+(x,t)\eta_{R}(x)\,dx
\leq \int_{\R^n}\left(\log u_1-\log u_2\right)_+(x,t)\left|\La \eta_{R}(x)\right|\,dx\quad \forall 0<t<T.
\end{equation}
By Theorem \ref{thm-behaviour-of-v-m-at-infty} there exists a constant $C_3>0$ such that
\begin{equation}\label{v-lambda-i-lower-bd-5}
v_{\lambda_i}(x)\ge C_3|x|^{-2}\quad\forall |x|\ge 1, i=1,2.
\end{equation}
By \eqref{eq-trapping-solution-u-i-between-V-lambda-1-and-V-lambda-2} and \eqref{v-lambda-i-lower-bd-5},
\begin{equation}\label{parabolic-solns-uniform-lower-bd}
u_i(x,t)\ge (T-t)^{\alpha}\cdot C_3\left((T-t)^{\beta}|x|\right)^{-2}=C_3(T-T_1)|x|^{-2}\quad\forall |x|\ge (T-T_1)^{-\beta}, 0<t\le T_1<T.
\end{equation}
By \eqref{eq-process-for-L-1-contraction-after-letting-epsilon-to-zero-1}, \eqref{parabolic-solns-uniform-lower-bd} and the same argument as the
proof of Lemma 2.5 of \cite{Hu2} for any $T_1\in (0,T)$ we get $u_1\leq u_2$ in $\left(\R^n\bs\{0\}\right)\times(0,T_1)$. Hence \eqref{u1<u2-ineqn} holds.

If $u_{0,1}=u_{0,2}$ and both $u_1$, $u_2$ are solutions of \eqref{eq-log-diffusion-equation-in-R-to-n-spacse-bs-zero-324r5}  in $\left(\R^n\bs\{0\}\right)\times(0,T)$ which satisfy \eqref{eq-trapping-solution-u-i-between-V-lambda-1-and-V-lambda-2}, then we also have $u_2\leq u_1$ in $\left(\R^n\bs\{0\}\right)\times(0,T)$. Hence $u_1=u_2$ in $\left(\R^n\bs\{0\}\right)\times(0,T)$ and the theorem follows.
\end{proof}

\begin{proof}[\textbf{Proof of Theorem \ref{parabolic-singular-limit-thm1}}:]
Let $m_0\in (0,\2{m}_0)$ by given by Lemma \ref{lemma-upper-bound-of-radially-symmetric-solution-1}. Then by \eqref{self-similar-soln-defn1} and Lemma \ref{lemma-upper-bound-of-radially-symmetric-solution-1}, for any  $x\in\R^n\setminus\{0\}$, $0<t<T$, $0<m<m_0$, $i=1,2$,
\begin{align}
& \lambda_i^{-\frac{1}{(1-m)\beta}}|x|^{-\frac{\alpha_m}{\beta}} \leq V_{\lambda_i}^{(m)}(x,t)\leq \lambda_i^{-\frac{1}{(1-m)\beta}}|x|^{-\frac{\alpha_m}{\beta}} \exp\left(C_m\lambda_i^{\frac{1}{\beta}}T|x|^{\frac{1}{\beta}}\right)\label{v-lambda-i-upper-lower-bd3}\\
\Rightarrow \quad &\underline{\lambda}_i|x|^{-\frac{\alpha_m}{\beta}} \leq V_{\lambda_i}^{(m)}(x,t)\leq \overline{\lambda}_i|x|^{-\frac{\alpha_m}{\beta}}\exp\left(\overline{C}_0\lambda_i^{\frac{1}{\beta}}T|x|^{\frac{1}{\beta}}\right)
\label{eq-upper-and-lower-bound-of-self-similar-sols-34}
\end{align}
where $C_m$ is given by \eqref{eq-constant-C-0-of-upper-bound-of-g-lambda} and
\begin{equation*}
\overline{\lambda}_i=\max\left(\lambda_i^{-\frac{n}{2\beta}},\lambda_i^{-\frac{1}{\beta}}\right), \qquad  \underline{\lambda}_i=\min\left(\lambda_i^{-\frac{n}{2\beta}},\lambda_i^{-\frac{1}{\beta}}\right)\qquad \mbox{and}\qquad \overline{C}_0=\frac{n(n-2)(2\beta+1)}{2\beta}.
\end{equation*}
By \eqref{eq-compare-between-radially-symmetric-sols-and-sollutions} and \eqref{eq-upper-and-lower-bound-of-self-similar-sols-34},
\begin{equation}\label{eq-upper-and-lower-bound-of-u-m-by-self-similar-sols-34}
\underline{\lambda}_1\min\left(|x|^{-\frac{n}{2}\left(2+\frac{1}{\beta}\right)},|x|^{-\left(2+\frac{1}{\beta}\right)}\right)\leq u^{(m)}(x,t)\leq \overline{\lambda}_2\max\left(|x|^{-\frac{n}{2}\left(2+\frac{1}{\beta}\right)},|x|^{-\left(2+\frac{1}{\beta}\right)}\right)\exp\left(\overline{C}_0\lambda_2^{\frac{1}{\beta}}T|x|^{\frac{1}{\beta}}\right)
\end{equation}
holds for any  $x\in \R^n\setminus\{0\}$, $0<t<T$ and  $0<m<m_0$.
Let $\{m_i\}_{i=1}^{\infty}\subset(0,m_0)$ be a sequence of positive numbers such that $m_i\to 0$ as $i\to\infty$. By \eqref{eq-upper-and-lower-bound-of-u-m-by-self-similar-sols-34} the equation \eqref{fde} for the sequence $\{u^{(m_i)}\}_{i=1}^{\infty}$ is uniformly parabolic on every compact subset of 
$(\R^n\setminus\{0\})\times(0,T)$. By the Schauder estimates for parabolic equations \cite{LSU}, the sequence $u^{(m_i)}(x,t)$ is equi-bounded in $C^{2+\theta,1+\frac{\theta}{2}}(K)$ for some $\theta\in(0,1)$ for any  compact subset $K$ of $(\R^n\bs\{0\})\times(0,T)$. Hence by the Ascoli theorem and a diagonalization argument the sequence $u^{(m_i)}(x,t)$ has a subsequence which we may assume without loss of generality to be the sequence itself that converges uniformly in $C^{2+\theta,1+\frac{\theta}{2}}(K)$ for any  compact subset $K$ of $(\R^n\bs\{0\})\times(0,T)$ as $i\to\infty$ to a positive function  $u(x,t)\in C^2(\R^n\setminus\{0\})$ 
which by \eqref{eq-compare-between-radially-symmetric-sols-and-sollutions} and Theorem \ref{elliptic-singular-limit-thm} satisfies
\eqref{eq-solution-u-trapped-by-V-sub-is}.

Putting $m=m_i$ in \eqref{v-lambda-i-upper-lower-bd3} and letting $i\to\infty$, by Theorem \ref{elliptic-singular-limit-thm},
\begin{equation}\label{eq-upper-lower-of-solution-u-of-log-diffusion}
\lambda_i^{-\frac{1}{\beta}}|x|^{-\frac{\alpha}{\beta}} \leq V_{i}(x,t)\leq \lambda_i^{-\frac{1}{\beta}}|x|^{-\frac{\alpha}{\beta}}\exp\left(\2{C}_0\lambda_i^{\frac{1}{\beta}}T|x|^{\frac{1}{\beta}}\right) \qquad \forall x\in\R^n\setminus\{0\},0<t<T,i=1,2.
\end{equation} 
By \eqref{eq-upper-and-lower-bound-of-u-m-by-self-similar-sols-34} and the mean value theorem for any $(x,t)\in (\R^n\setminus\{0\})\times (0,T)$ there exists $\xi_i(x,t)\in (0,m_i]$ such that
\begin{equation}\label{eq-aligned-difference-between-u-m-on-x-t-over-m-and-log-u-on-x-t-to=zoer-as-i-to-infty}
\begin{aligned}
&\left|\frac{u^{(m_i)}(x,t)^{m_i}-1}{m_i}-\log u(x,t)\right|\\
=&\left|e^{\xi_i(x,t)\log u^{(m_i)}(x,t)}\log u^{(m_i)}(x,t)-\log u(x,t)\right|\\
\le& e^{\xi_i(x,t)\log u^{(m_i)}(x,t)}\left|\log u^{(m_i)}(x,t)-\log u(x,t)\right|+\left|e^{\xi_i(x,t)\log u^{(m_i)}(x,t)}-1\right|\cdot\left|\log u(x,t)\right|\\
\to& 0 \qquad \qquad \mbox{uniformly on every compact subset of $(\R^n\bs\{0\})\times(0,T)$ as $i\to\infty$}.
\end{aligned}
\end{equation}
Hence putting $m=m_i$ in \eqref{fde}  and letting $i\to\infty$, by \eqref{eq-aligned-difference-between-u-m-on-x-t-over-m-and-log-u-on-x-t-to=zoer-as-i-to-infty} $u$ satisfies \eqref{log-diffusion-eqn}. It remains to prove that $u$ has initial value $u_0$. For any $\psi\in C_0^{\infty}\left(\R^n\bs\{0\}\right)$, we choose constants $R_2>R_1>0$ such that $\supp\psi\subset B_{R_2}\bs B_{R_1}$.
Then
\begin{align}
&\left|\int_{\R^n\bs\{0\}}u^{(m_i)}(x,t)\psi(x)\,dx-\int_{\R^n\bs\{0\}}u_{0,m}(x)\psi(x)\,dx\right|\notag\\
=&\left|\int_{0}^{t}\int_{\R^n\bs\{0\}}u^{(m_i)}_t(x,s)\psi(x)\,dxds\right|
=\left|\int_{0}^{t}\int_{\R^n\bs\{0\}}\left(\frac{u^{(m_i)}(x,s)^{m_i}-(T-s)^{m_i\alpha_{m_i}}}{m_i}\right)\cdot\La\psi(x)\,dxds\right|\notag\\
\le&\left\|\La\psi\right\|_{L^{\infty}}\int_0^t\int_{B_{R_2}\bs B_{R_1}}\left(E_1+E_2\right)\,dxds\qquad \qquad \quad\forall 0<t<T
\label{eq-aligned-eq-for-u-m-andu-sub-m-zero-diffrerence-after-before-letting-m-to-zero}
\end{align}
where
\begin{equation*}
E_k=\left|\frac{V_{\lambda_k}^{(m)}(x,s)^{m_i}-(T-s)^{m_i\alpha_{m_i}}}{m_i}\right|\quad\forall k=1,2.
\end{equation*}
Since
\begin{equation*}
E_k=\left|(T-s)^{m_i\alpha_{m_i}}\right|\left|\frac{v^{(m)}_{\lambda_k}(\left(T-s\right)^{\beta}|x|)^{m_i}-1}{m_i}\right| 
\to\left|\log v_{\lambda_k}\left(\left(T-s\right)^{\beta}x\right)\right|\quad\mbox{ uniformly on }(B_{R_2}\bs B_{R_1})\times(0,T_1)
\end{equation*} 
for any $0<T_1<T$ as $i\to\infty$, letting $i\to\infty$ in \eqref{eq-aligned-eq-for-u-m-andu-sub-m-zero-diffrerence-after-before-letting-m-to-zero}, 
\begin{align}
&\left|\int_{\R^n\bs\{0\}}u(x,t)\psi(x)\,dx-\int_{\R^n\bs\{0\}}u_{0}(x)\psi(x)\,dx\right|\notag\\ 
\le&\left\|\La\psi\right\|_{L^{\infty}}\int_0^t\int_{B_{R_2}\bs B_{R_1}}\left(\left|\log v_{\lambda_1}\left((T-s)^{\beta}x\right)\right|+\left|\log v_{\lambda_2}\left((T-s)^{\beta}x\right)\right|\right)\,dxds\notag\\
\le& C_1\left\|\La\psi\right\|_{L^{\infty}}t \qquad \qquad \qquad \forall 0<t\leq T/2
\label{eq-aligned-eq-for-u-andu-sub-zero-diffrerence-after-letting-m-to-zero}
\end{align}
where
\begin{equation*}
C_1=\max_{\left(\frac{T}{2}\right)^{\beta}R_1\leq|y|\leq T^{\beta}R_2}\left|\log v_{\lambda_1}(y)\right|+\max_{\left(\frac{T}{2}\right)^{\beta}R_1\leq|y|\leq T^{\beta}R_2}\left|\log v_{\lambda_2}(y)\right|.
\end{equation*}
Letting $t\to 0$ in \eqref{eq-aligned-eq-for-u-andu-sub-zero-diffrerence-after-letting-m-to-zero},
\begin{equation}\label{eq-weak-convergence-in-L-1-of-u-to-initial-data}
\lim_{t\to 0}\int_{\R^n\bs\{0\}}u(x,t)\psi(x)\,dx=\int_{\R^n\bs\{0\}}u_{0}(x)\psi(x)\,dx \qquad \forall\psi\in C^{\infty}_0(\R^n\bs\{0\}).
\end{equation}
By \eqref{eq-weak-convergence-in-L-1-of-u-to-initial-data}, any sequence $\left\{t_{k}\right\}_{k=1}^{\infty}$ converging to $0$ as $k\to\infty$ will have a  subsequence $\{t_{k_l}\}_{l=1}^{\infty}$ such that $u(x,t_{k_l})$ converges to $u_0(x)$ for a.e.  $x\in\R^n\setminus\{0\}$ as $l\to \infty$. Then by the Lebesgue Dominated Convergence Theorem,
\begin{equation*}
\lim_{l\to\infty}\int_{R_1\le|x|\le R_2}\left|u(x,t_{k_l})-u_0(x)\right|\,dx=0\quad\forall R_2>R_1>0.
\end{equation*}
Since the sequence $\left\{t_k\right\}_{k=1}^{\infty}$ is arbitrary, $u(\cdot,t)$ converges to $u_0$ in $L_{loc}^1(\R^n)$ as $t\to 0$. Hence $u$ has initial value $u_0$. Thus by Theorem \ref{singular-log-diffusion-eqn-uniqueness-thm} $u$ is the unique solution  of \eqref{eq-cases-aligned-problem-of-parabolic-case-after-m-to-zero}.
Hence $u^{(m)}$ converges uniformly in $C^{2+\theta,1+\frac{\theta}{2}}(K)$  for some constant $\theta\in (0,1)$ and any  compact subset $K$ of $\left(\R^n\bs\left\{0\right\}\right)\times(0,T)$ to the solution $u$ of \eqref{eq-cases-aligned-problem-of-parabolic-case-after-m-to-zero} as $m\to 0^+$ and the theorem follows. 
\end{proof}

\noindent {\bf Acknowledgement:} We would like to thank the referee for the many valuable comments and suggestions on the paper. 
K.M.~Hui was partially supported by the Ministry of Science and Technology of Taiwan (MOST 105-2115-M-001-002). 
Sunghoon Kim was supported by the National Research Foundation of Korea (NRF) grant funded by the Korea government (MSIP, no. 2015R1C1A1A02036548). Sunghoon Kim was also supported by the Research Fund 2016 of The Catholic University of Korea.

\end{document}